\documentclass{svproc}
\usepackage{url}

\usepackage{helvet}         
\usepackage{type1cm}        
\usepackage{makeidx}         
\usepackage{graphicx}        
\usepackage{multicol}        
\usepackage[bottom]{footmisc}
\usepackage{bm}
\usepackage{mathtools}
\usepackage{xcolor}
\RequirePackage[colorlinks,citecolor=blue,urlcolor=blue]{hyperref}

\makeindex             

\usepackage{latexsym}
\usepackage{dsfont}
\usepackage{bbm}
\usepackage{amssymb}
\usepackage{amsmath}

\newtheorem{Theorem}{Theorem}[section]
\newtheorem{Lemma}[Theorem]{Lemma}

\newtheorem{Prop}[Theorem]{Proposition}
\newtheorem{Rem}[Theorem]{Remark}

\numberwithin{equation}{section}
\numberwithin{Theorem}{section}

\def\cL{\mathcal{L}}

\def\cZ{\mathcal{Z}}

\def\Erw{\mathbb{E}}

\def\N{\mathbb{N}}
  
\def\Prob{\mathbb{P}} 

\def\R{\mathbb{R}}

\def\Var{\mathbb{V}{\rm ar\,}}

\def\sfm{\mathsf{m}}

\def\bfE{\text{\rm\bfseries E}}
\def\bfe{\text{\rm\bfseries e}}

\def\bfP{\text{\rm\bfseries P}}

\def\vph{\varphi}

\def\1{\mathbf{1}}
\def\3{{\ss}}

\def\eqdist{\stackrel{d}{=}}
\def\iprob{\stackrel{\Prob}{\to}}
\def\idist{\stackrel{d}{\longrightarrow}}

\def\weakly{\stackrel{w}{\longrightarrow}}

\def\wh{\widehat}

\def\ovl{\overline}

\def\geq{\geqslant}

\def\ssy{\scriptscriptstyle}

\def\GWP{$\mathsf{GWP}$}
\def\GWPRE{$\mathsf{GWPRE}$}

\def\Geom{\mathsf{Geom}}
\def\LF{\mathsf{LF}}
\def\CLF{\mathsf{CLF}}
\def\ML{\mathsf{MiLe}}
\def\PF{\mathsf{PF}}
\def\CPF{\mathsf{CPF}}
\def\Sib{\mathsf{Sib}\hspace{.6pt}}
\def\GSib{\mathsf{GSib}\hspace{.6pt}}

\def\le{\leqslant}
\def\ge{\geqslant}

\def\cblue{\color{blue}}

\allowdisplaybreaks[4] 

\begin{document}
\mainmatter              
\title{Power-fractional distributions and branching processes}
\titlerunning{Power-fractional distributions and branching processes}  
%
\author{Gerold Alsmeyer\inst{1} \and Viet Hung Hoang\inst{2}}
\authorrunning{Gerold Alsmeyer and Viet Hung Hoang} 
%
\tocauthor{Ivar Ekeland, Roger Temam, Jeffrey Dean, David Grove,
Craig Chambers, Kim B. Bruce, and Elisa Bertino}
\institute{Inst.~Math.~Stochastics, Department
of Mathematics and Computer Science\\ University of M\"unster, Orl\'eans-Ring 10, D-48149  M\"unster, Germany\\
\email{gerolda@math.uni-muenster.de}
\and
Faculty of Fundamental Science, Industrial University of Ho Chi Minh City,\\ No. 12 Nguyen Van Bao, Ward 4, Go Vap District,\\ Ho Chi Minh City, Vietnam\\
\email{hoangviethung@iuh.edu.vn}
}

\maketitle              

\begin{abstract}
In branching process theory, linear-fractional distributions are commonly used to model individual reproduction, especially when the goal is to obtain more explicit formulas than those derived under general model assumptions. In this article, we explore a generalization of these distributions, first introduced by Sagitov and Lindo, which offers similar advantages. We refer to these as power-fractional distributions, primarily because, as we demonstrate, they exhibit power-law behavior. Along with a discussion of their additional properties, we present several results related to the Galton-Watson branching process in both constant and randomly varying environments, illustrating these advantages. The use of power-fractional distributions in continuous time, particularly within the framework of Markov branching processes, is also briefly addressed.

\end{abstract}
\section{Introduction}\label{sec:1}
It is a well-known fact in the theory of branching processes that the linear-fractional distribution $\LF(a,b)$ with probability generating function (pgf)
\begin{equation*}
f(s)\ =\ 1-\left[\frac{a}{1-s}+b\right]^{-1}
\end{equation*}
for parameters $a,b>0$ with $a+b\ge 1$, is stable under iterations of $f$, namely
\begin{equation}\label{eq:def LF gf}
\frac{1}{1-f^{n}(s)}\ =\ \frac{a^{n}}{1-s}+b_{n},\quad b_{n}\,:=\,b\sum_{k=0}^{n-1}a^{k}
\end{equation}
for each $n\in\N$, and so $f^{n}$, the $n$-fold iteration of $f$, equals the pgf of $\LF(a^{n},b_{n})$. But $f^{n}$ is also the pgf of $Z_{n}$ for each $n$ if $(Z_{n})_{n\ge 0}$ is a Galton-Watson branching process with one ancestor and offspring law $\LF(a,b)$. Therefore, all $Z_{n}$ have a linear-fractional distribution, namely $Z_{n}\eqdist\LF(a^{n},b_{n})$, where $\eqdist$ means equality in law. 

\vspace{.1cm}
Sagitov and Lindo in \cite[Section\,14.2]{SagitovLindo:16} introduced an extended class of possibly defective distributions on $\N_{0}$ that are stable under iterations of their pgf's. Any such distribution, here called \emph{power-fractional law} and abbreviated as $\PF(\theta,\gamma,a,b)$, involves two further parameters $\theta\in [-1,1]\backslash\{0\}$ and $\gamma\ge 1$. Its pgf has the form
\begin{gather}
f(s)\ =\ \gamma-\left[\frac{a}{(\gamma-s)^{\theta}}+b\right]^{-1/\theta},\quad s\in [0,\gamma).
\label{eq1:def PF gf}
\end{gather}
where $\theta,\gamma,a,b$ must satisfy the conditions of one of the cases (A1--3) shown in the Table \ref{tab:1.1}.
\begin{table}
\centering
\begin{tabular}{lllll}
\hline\noalign{\smallskip}
Case &$\theta$ &$\gamma$ & \multicolumn{2}{c}{$a,\,b\qquad$}  \\
\noalign{\smallskip}\hline\noalign{\smallskip}
(A1)\hspace{.4cm} &$\theta\in (0,1]$ &$\gamma=1$ &$a,b>0$ &$a+b\ge 1$\\[1mm]
(A2) &$\theta\in (0,1]$ &$\gamma>1$ &$a\in (0,1)$\hspace{.4cm} &$\gamma^{-\theta}\le\frac{b}{1-a}\le (\gamma-1)^{-\theta}$\\[1mm]
(A3) &$\theta\in [-1,0)$\hspace{.4cm} &$\gamma\ge 1$\hspace{.4cm} &$a\in (0,1)$ &$(\gamma-1)^{|\theta|}\le\frac{b}{1-a}\le\gamma^{|\theta|}$\\
\noalign{\smallskip}\hline\noalign{\smallskip}
\end{tabular}
\caption{The constraints on the parameters $\theta,\gamma,a$ and $b$}
\label{tab:1.1}
\end{table}
The linear-fractional laws $\LF(a,b)$ are obtained for $\theta=\gamma=1$, but the subcritical ones also for $\theta=1$ and $\gamma>1$ (see Comment (d) below). With $q$ denoting the minimal fixed point of $f$ in $[0,1]$, the last constraint in (A2) and (A3) may be restated as
\begin{align*}
b\,=\,\frac{1-a}{(\gamma-q)^{\theta}}\quad\text{and}\quad q\in\,\begin{cases} [0,1]&\text{if }\gamma>1,\\ [0,1)&\text{if }\gamma=1,\end{cases}
\end{align*}
see \cite[Def.\,14.1]{SagitovLindo:16}. One can also readily check that Eq.\,\eqref{eq1:def PF gf} is equivalent to
\begin{gather}\label{eq2:def PF gf}
\frac{1}{(\gamma-f(s))^{\theta}}\ =\ \frac{a}{(\gamma-s)^{\theta}}+b,\quad s\in [0,\gamma).
\end{gather}

\vspace{.1cm}
The parameter $\theta=0$ appears as a limiting case, namely,
\begin{gather}
f(s)=\gamma-(\gamma-q)^{1-a}(\gamma-s)^{a}\label{eq3:def PF gf}
\shortintertext{or, equivalently,}
\gamma-f(s)\ =\ (\gamma-q)^{1-a}(\gamma-s)^{a},\quad s\in [0,\gamma).
\label{eq4:def PF gf}
\end{gather}
This can be deduced from \eqref{eq1:def PF gf} by a continuity argument when observing that
\begin{align*}
\lim_{\theta\to 0}\left[\frac{a}{(\gamma-s)^{\theta}}+\frac{1-a}{(\gamma-q)^{\theta}}\right]^{-1/\theta}\  =\ (\gamma-q)^{1-a}(\gamma-s)^{a}.
\end{align*}
Here the parameters are $\gamma,q$ and $a$, and they must satisfy $a\in (0,1)$ and either $\gamma=1,\,q\in [0,1)$, or $\gamma>1,\,q\in [0,1]$. 

\vspace{.1cm}
We summarize the previous definitions (without repeating the constraints on $\gamma,a,b$ and $q$) as
\begin{gather}\label{eq:compact def of f(s)}
f(s)\ =\ 
\begin{cases}
\hfill\displaystyle\gamma\,-\,\frac{\gamma-s}{\big[a+b(\gamma-s)^{\theta}\big]^{1/\theta}}&\text{if }\theta\in (0,1],\\[4mm]
\hfill\gamma\,-\,(\gamma-q)^{1-a}(\gamma-s)^{a}&\text{if }\theta=0,\\[1mm]
\gamma\,-\,\big[a(\gamma-s)^{|\theta|}+b\big]^{1/|\theta|}&\text{if }\theta\in [-1,0)
\end{cases}
\end{gather}
and also note that the iterations of $f$ are easily computed with the help of Eqs.~\eqref{eq2:def PF gf} and \eqref{eq4:def PF gf} as
\begin{gather}\label{eq:iteration of f(s)}
f^{n}(s)\ =\ 
\begin{cases}
\hfill\displaystyle\gamma\,-\,\frac{\gamma-s}{\big[a^{n}+b_{n}(\gamma-s)^{\theta}\big]^{1/\theta}}&\text{if }\theta\in (0,1],\\[4mm]
\hfill\gamma\,-\,(\gamma-q)^{1-a^{n}}(\gamma-s)^{a^{n}}&\text{if }\theta=0,\\[1mm]
\gamma\,-\,\big[a^{n}(\gamma-s)^{|\theta|}+b_{n}\big]^{1/|\theta|}&\text{if }\theta\in [-1,0).
\end{cases}
\end{gather}
for any $n\ge 1$, where $b_{n}=b\sum_{k=0}^{n-1}a^{k}$ as in \eqref{eq:def LF gf}.

\vspace{.1cm}
Due to the leading parameter $\theta$, Sagitov and Lindo \cite{SagitovLindo:16} coined the term \emph{theta-branching process} for a Galton-Watson process (\GWP) with power-fractional~off\-spring law. 

\vspace{.1cm}
We continue with some notable facts and supplements in connection with the above definitions:
\begin{itemize}\itemsep3pt
\item[(a)] As pointed out in \cite[Sect.\,1]{LindoSagitov:15}, Eq.'s\,\eqref{eq2:def PF gf} and \eqref{eq4:def PF gf} are both special cases of a general functional equation, namely
\begin{equation}\label{eq:general functional equation}
H(f(s))\ =\ aH(s)\,+\,H(f(0))
\end{equation}
where $H(s)=H_{\theta,\gamma}(s)=(\gamma-s)^{-\theta}-\gamma^{-\theta}$ in \eqref{eq2:def PF gf},\,\eqref{eq3:def PF gf} and $H(s)=H_{0,\gamma}(s)=\log\gamma-\log(\gamma-s)$ in \eqref{eq4:def PF gf}. Regarding the minimal fixed point $q$ of $f$, it follows that 
\begin{equation*}
H(q)\ =\ \begin{cases}
\displaystyle\frac{H(f(0))}{1-a}&\text{if }a\in (0,1),\\[2mm]
\hfil\infty&\text{if }a\ge 1\text{ and thus }\theta\in (0,1],\,\gamma=1,\,q=1.
\end{cases}
\end{equation*}
\item[(b)] Since $q$ equals the smallest fixed point of $f$ in $[0,1]$, it is also the extinction probability of the associated branching process and therefore the notational choice in accordance with a common convention in the branching process literature.
\item[(c)] We always have that $f(1)\le 1$, and $f(1)=1$ holds iff one of the cases shown in Table \ref{tab:1.2} occurs. Notably, $b$ is a unique function of $a$ (and $\theta,\gamma$) in the cases (A.2.1) and (A.3.1), giving
\begin{gather}
f(s)\,=\,\gamma\,-\,\bigg[\frac{a}{(\gamma-s)^{\theta}}+\frac{1-a}{(\gamma-1)^{\theta}}\bigg]^{1/\theta}\label{eq5:def PF gf}
\shortintertext{and}
f(s)\,=\,\gamma\,-\,\big[a(\gamma-s)^{|\theta|}+(1-a)(\gamma-1)^{|\theta|}\big]^{1/|\theta|},
\label{eq6:def PF gf}
\end{gather}
respectively.
\end{itemize}
\begin{table}
\vspace{-.3cm}
\centering
\begin{tabular}{lllll}
\hline\noalign{\smallskip}
Case &$\theta$ &$\gamma$ & $a$ &$b$  \\
\noalign{\smallskip}\hline\noalign{\smallskip}
(A1)\hspace{.4cm} &$\theta\in (0,1]$ &$\gamma=1$ &$a>0$ &$b>0,\,b\ge 1-a$\\[1mm]
(A2.1) &$\theta\in (0,1]$ &$\gamma>1$ &$a\in (0,1)$\hspace{.4cm} &$b=(1-a)(\gamma-1)^{-\theta}$\\[1mm]
(A3.1) &$\theta\in [-1,0)$\hspace{.4cm} &$\gamma\ge 1$\hspace{.4cm} &$a\in (0,1)$ &$b=(1-a)(\gamma-1)^{|\theta|}$\\
\noalign{\smallskip}\hline\noalign{\smallskip}
\end{tabular}
\caption{The constraints on the parameters $\theta,\gamma,a$ and $b$ that are required for $f(1)=1$.}
\label{tab:1.2}
\end{table}
\vspace{-.6cm}
\begin{itemize}\itemsep3pt
\item[] Thus, a power-fractional law can be defective or, equivalently, assign mass to $\infty$, in which case the \GWP\ with this offspring law exhibits explosion in finite time.
\item[(d)] We have already pointed out that the class of linear-fractional distributions is obtained when $\theta=\gamma=1$, i.e.
$$ \LF(a,b)\,=\,\PF(1,1,a,b)\quad\text{for all }a,b>0\text{ such that }a+b\ge 1. $$
On the other hand, the choice $\theta=1$ and $\gamma>1$ obviously also leads to linear-fractional laws. In fact, if $\LF(a,b)$ is subcritical, meaning $a>1$, then the associated pgf $f$ has the two fixed points $1$ and $\wh{q}=b^{-1}(a+b-1)>1$ (in the supercritical case $<1$ and the extinction probability), and one can readily verify that $f$ satisfies \eqref{eq2:def PF gf} not only for $\gamma=1$, but also for $\gamma=\wh{q}$. Namely,
\begin{gather*}
\frac{1}{\wh{q}-f(s)}\ =\ \frac{a^{-1}}{\wh{q}-s}\,+\,\frac{b}{a},\quad s\in [0,\wh{q}\,),
\shortintertext{and therefore}
\PF(1,1,a,b)\ =\ \PF\left(1,\frac{a+b-1}{b},\frac{1}{a},\frac{b}{a}\right)\quad\text{for all }a>1,\,b>0.
\end{gather*}
The reader can check that this aligns with (A2.1) by substituting the pair $(a^{-1},b^{-1}(a+b-1))$ for $(a,\gamma)$. Therefore, we conclude that the parametrization of a subcritical linear-fractional law within the class of power-fractional laws is not unique.
\item[(e)] Let us also mention that, if $\theta\ne 0$ and $\psi(s):=\left[a+b(\gamma-s)^{\theta}\right]^{-1/\theta}$, then
$$ f'(s)\ =\ a\psi(s)^{\theta+1}\quad\text{and}\quad f''(s)\ =\ ab(\theta+1)(\gamma-s)^{\theta-1}\psi(s)^{2\theta+1}. $$
The resulting values for $\sfm=f'(1)$ (the mean number of offspring und thus criticality parameter) and $f''(1)$ in the three proper cases (A1), (A2.1) and (A3.1) are listed in Table \ref{tab:1.3} below, keeping in mind that in the last two cases $b$ is uniquely determined by the other parameters. We see that the critical case $\sfm=1$ occurs only if $\theta>0$ and $a=1$, and the supercritical case only if $\theta>0$ and $a\in (0,1)$. This can be viewed as some evidence for Case (A1) providing the most interesting subclass of power-fractional distributions in the context of branching processes.
\begin{table}
\centering
\begin{tabular}{lllll}
\hline\noalign{\smallskip}
Case &$\theta$ &$\gamma$ & $\sfm=f'(1)$\hspace{.3cm} & $f'\hspace{-1pt}'(1)$  \\
\noalign{\smallskip}\hline\noalign{\smallskip}
(A1) &$\theta\in (0,1)$ &$\gamma=1$ &$a^{-1/\theta}$ &$\infty$\\[1mm]
(A1) &$\theta=1$ &$\gamma=1$ &$a^{-1}$ &$2a^{-2}b$\\[1mm]
(A2.1) &$\theta\in (0,1]$ &$\gamma>1$ &$a$ &$a(1-a)(1+\theta)/(\gamma-1)$\\[1mm]
(A3.1) &$\theta\in [-1,0)$ &$\gamma>1$ &$a$ &$a(1-a)(1-|\theta|)/(\gamma-1)$\\[1mm]
(A3.1)\hspace{.3cm} &$\theta\in [-1,0)$\hspace{.3cm} &$\gamma=1$\hspace{.3cm} &$a^{1/|\theta|}$ &$0$\\
\noalign{\smallskip}\hline\noalign{\smallskip}
\end{tabular}
\caption{First and second derivative of $f$ at 1 for the proper cases when $f(1)=1$.}
\label{tab:1.3}
\end{table}
\item[(f)] The \emph{Sibuya distribution} $\Sib(a)$ for $a\in (0,1)$, named after Sibuya \cite{Sibuya:79}, has pgf
\begin{gather}
f(s)\ =\ 1-(1-s)^{a},\quad s\in [0,1],\label{eq:pgf Sib}
\intertext{support $\N$ and probability mass function}
p_{n}\,=\,(-1)^{n+1}\binom{a}{n}\text{ for }n\ge 1.\label{eq:pmf Sib}
\end{gather}
It appears here as a particular power-fractional law with $\theta=0$ and, furthermore, $\gamma=1$, $a\in (0,1)$ and $q=0$, see \eqref{eq3:def PF gf}. Note also that it has mean $\sfm=\infty$. 

\vspace{.1cm}
The reason for explicitly highlighting this subclass is that, in terms of pgf's, every power-fractional law with parameter $\theta\in (0,1)$ can be expressed as a conjugation of a linear-fractional law and a Sibuya law. More specifically,
\begin{equation}\label{eq:conjugation rule with Sibuya}
h\circ f\,=\,g\circ h
\end{equation}
if, for $\theta\in (0,1)$, $f,g$ and $h$ are the pgf's of $\PF(\theta,1,a,b)$, $\LF(a,b)$ and $\Sib(\theta)$, respectively. This is easily checked and means that
\begin{equation}\label{eq2:conjugation rule with Sibuya}
\sum_{k=1}^{S}X_{k}\ \eqdist\ \sum_{k=1}^{Y}S_{k}
\end{equation}
for independent random variables $X,Y$ and $X_{n},S_{n}$, $n=1,2,\ldots$, such that the law of $X,X_{1},X_{2},\ldots$ is $\PF(\theta,1,a,b)$, the law of $Y$ is $\LF(a,b)$, and the law of $S,S_{1},S_{2},\ldots$ is $\Sib(\theta)$. The identity remains valid for $\theta=1$, because $\PF(1,1,a,b)=\LF(a,b)$ and $\Sib(1)$, with pgf $f(s)=1-(1-s)=s$, is the point measure at 1, so $S=1$ a.s.

\item[(g)] By taking a convex combination of a Sibuya distribution with the Dirac measure at 0, we stay in the subclass of power-fractional laws with parameter $\theta=0$. More specifically, defining the \emph{generalized Sibuya distribution} $\GSib(a,q)=(p_{n})_{n\ge 0}$ with parameter $(a,q)\in (0,1)\times [0,1)$ as
\begin{gather*}\label{eq:def GSib}
\GSib(a,q)\ =\ (1-(1-q)^{1-a})\delta_{0}\,+\,(1-q)^{1-a}\Sib(a),
\shortintertext{which has pgf}
f(s)\ =\ 1-(1-q)^{1-a}(1-s)^{a},\quad s\in [0,1]\label{eq:pgf GSib},
\end{gather*}
we obtain the power-fractional law with $\theta=0,\,\gamma=1,\,a\in (0,1),$ and general $q$, see again \eqref{eq3:def PF gf}. Plainly, $\GSib(a,0)=\Sib(a)$.

\vspace{.1cm}
Finally, we mention that the conjugation relation \eqref{eq:conjugation rule with Sibuya} remains valid if $h$ is the pgf of a generalized Sibuya distribution. Thus, \eqref{eq2:conjugation rule with Sibuya} also holds if the $S,S_{1},S_{2},\ldots$ have common law $\GSib(a,q)$.
\end{itemize}

After this general introduction to power-fractional distributions, we will focus for the remainder of this article on case (A1), and thus adopt the following

\vspace{.2cm}\noindent
\textit{Standing Assumption}: $\theta \in (0,1]$ and $\gamma=1$.

\vspace{.2cm}\noindent 
This assumption not only represents the most natural extension of the linear-frac\-tional case, but is also the only one that encompasses critical, subcritical,~and supercritical laws. Furthermore, some of the results we will present do not apply to other cases. With $\gamma = 1$ fixed henceforth, we will omit it from the list~of parameters of a power-fractional law, simply writing $\PF(\theta, a, b)$ to refer to $\PF(\theta,1,a,b)$.


\section{Power-law behavior}

This section presents some basic properties of a power-fractional law $(p_{n})_{n\ge 0}=\PF(\theta,a,b)$ for $\theta\in (0,1]$ and $a,b>0$ with $a+b\ge 1$. Our main result will demonstrate that, for $\theta<1$, this law is essentially a power law, meaning that $p_{n}$ behaves like $n^{-(2+\theta)}$ as $n\to\infty$. The precise statement of this result can be found in Theorem \ref{thm:asymptotics pn} below. The following proposition by Sagitov and Lindo \cite[Prop. and Cor.\,14.1]{SagitovLindo:16} is not only essential for its proof, but also of interest in its own right. In particular, it shows that the function $f$ defined by \eqref{eq1:def PF gf} for $\gamma=1$ and $\theta\in (0,1]$ is indeed the pgf~of a proper probability distribution $(p_{n})_{n\ge 0}$ on $\N_{0}$. Moreover, it reveals the monotonicity of the $p_{n}$ and provides a formula  for them, see \eqref{eq2:formula pn}. This formula plays a crucial role in the proof of Theorem \ref{thm:asymptotics pn}.

\begin{Prop}
Fixing $\theta\in (0,1]$ and $a,b>0$ with $a+b\ge 1$, we define nonnegative numbers $c_{n,i}$ for $n\ge 2$ and $i=0,\ldots n$ as follows:
\begin{gather}
c_{2,1}=1+\theta\quad\text{and}\quad c_{n,0}=c_{n,n}=0\quad\text{for all }n\ge 2\label{eq1:c_n,i}
\intertext{and, recursively,}
c_{n,i}=(n-2-i\theta)c_{n-1,i}+(1+i\theta)c_{n-1,i-1}\label{eq2:c_n,i}
\end{gather}
for $n\ge 3$ and $i=1,\ldots,n-1$. Then $f(s)=\sum_{n\ge 0}p_{n}s^{n}$, where
\begin{gather}
p_{0}\,=\,1-\frac{1}{(a+b)^{1/\theta}},\quad p_{1}\,=\,\frac{a}{(a+b)^{1+1/\theta}}\label{eq1:formula pn}
\shortintertext{and}
p_{n}\,=\,\frac{p_{1}}{n!}\sum_{i=1}^{n-1}\left(\frac{a}{a+b}\right)^{i}c_{n,i}\label{eq2:formula pn}
\end{gather}
for $n\ge 2$. Moreover, $p_{n}\ge p_{n+1}$ for all $n\ge 1$.
\end{Prop}

Note that, since $f(1)=1$ (as a consequence of \eqref{eq1:def PF gf}) and all $p_{n}$ are nonnegative, the proposition does indeed establish that $f$ is the pgf~of a probability distribution. Note also that the $c_{n,i}$ depend on the parameter $\theta$, but not on the parameters $a,b$.

\begin{proof}
Putting $\rho(s)=(a+b(1-s)^{\theta})^{-1/\theta}$, one can readily check that
$$ f(s)\ =\ 1-\rho(s)(1-s). $$
Since $\rho'(s)=b\rho(s)^{\theta+1}(1-s)^{1-\theta}$, it follows that
\begin{align}
f'(s)\,&=\,\rho(s)-b\rho(s)^{\theta+1}(1-s)^{\theta}\,=\,a\rho(s)^{\theta+1},
\label{eq:fprime}\\
f''(s)\,&=\,ab(\theta+1)\rho(s)^{2\theta+1}(1-s)^{\theta-1}.\label{eq:fdoubleprime}
\end{align}
We claim and prove by induction that, for all $n\ge 2$,
\begin{equation}\label{eq:formula derivative of f}
f^{(n)}(s)\ =\ \sum_{i=1}^{n-1}ab^{i}c_{n,i}\rho(s)^{(i+1)\theta+1}(1-s)^{i\theta-n+1}.
\end{equation}
By \eqref{eq:fdoubleprime}, the claim holds true for $n=2$. Assuming it holds for an arbitrary fixed $n\ge 2$ (inductive hypothesis), we differentiate to obtain:
\begin{align*}
f^{(n+1)}(s)\ &=\ \sum_{i=1}^{n-1}ab^{i}c_{n,i}b\big((i+1)\theta+1\big)\rho(s)^{(i+2)\theta+1}(1-s)^{(i+1)\theta-n}\\
&\quad+\ \sum_{i=1}^{n-1}ab^{i}c_{n,i}(n-i\theta-1)\rho(s)^{(i+1)\theta+1}(1-s)^{i\theta-n}\\
&=\ \sum_{i=2}^{n}ab^{i-1}c_{n,i-1}b\big(i\theta+1\big)\rho(s)^{(i+1)\theta+1}(1-s)^{i\theta-n}\\
&\quad+\ \sum_{i=1}^{n-1}ab^{i}c_{n,i}(n-i\theta-1)\rho(s)^{(i+1)\theta+1}(1-s)^{i\theta-n}\\
=\ \sum_{i=1}^{n}&ab^{i}\big((i\theta+1)c_{n,i-1}+(n-i\theta-1)c_{n,i}\big)\rho(s)^{(i+1)\theta+1}(1-s)^{i\theta-n}\\
=\ \sum_{i=1}^{n}&ab^{i}c_{n+1,i}\rho(s)^{(i+1)\theta+1}(1-s)^{i\theta-n},
\end{align*}
where $c_{n,0}=c_{n,n}=0$ was used in the penultimate line, and the recursive definition of the $c_{n,i}$ applied in the final one. This completes the proof of \eqref{eq:formula derivative of f}, and the values of $p_{n}=f^{(n)}(0)/n!$ are now easily computed as stated in the lemma. We omit further details.\qed
\end{proof}

Given sequences $(a_{n})_{n\ge 0}$ and $(b_{n})_{n\ge 0}$ of real numbers, we write $a_{n}\simeq b_{n}$ if $a_{n}/b_{n}\to 1$, and $a_{n}\asymp b_{n}$ if $c_{1}\le a_{n}/b_{n}\le c_{2}$ for some $c_{1},c_{2}>0$ and all sufficiently large $n$.

\begin{Theorem}\label{thm:asymptotics pn}
Let $(p_{n})_{n\ge 0}=\PF(\theta,a,b)$ for $\theta\in (0,1)$ and $a,b>0$ such that $a+b\ge 1$. Then
\begin{equation}\label{eq:magnitude behavior pn}
p_{n}\ \asymp\ n^{-(2+\theta)}\quad\text{as }n\to\infty.
\end{equation}
Furthermore, if $a(a+b)^{-1}\le\theta(1+2\theta)^{-1}$, then the sequence $(n(n-1)p_{n})_{n\ge 2}$ is nonincreasing and
\begin{equation}\label{eq:exact behavior pn}
p_{n}\ \simeq\ cn^{-(2+\theta)}\quad\text{as }n\to\infty,
\end{equation}
where $c=a^{-(\theta+1)/\theta}b(\theta+1)/\Gamma(1-\theta)$.
\end{Theorem}

\begin{proof}
Note that $f''(s)=\sum_{n\ge 2}n(n-1)p_{n}s^{n}$ and recall from Remark (e) in the Introduction that, in the given situation, $f''(s)$ satisfies
\begin{align*}
f''(s)\ =\ ab(\theta+1)(1-s)^{\theta-1}\psi(s)^{2\theta+1}
\end{align*}
with $\psi(s)=\big[a+b(1-s)^{\theta}]^{-1/\theta}$. Since $\psi(s)\simeq a^{-1/\theta}$ as $s\uparrow 1$, we see that
\begin{align*}
f''(s)\ \simeq\ a^{-(\theta+1)/\theta}b(\theta+1)(1-s)^{\theta-1}\quad\text{as }s\uparrow 1.
\end{align*}
Therefore, if the $n(n-1)p_{n}$ are nonincreasing, then
$$ n(n-1)p_{n}\ \simeq\ \frac{a^{-(\theta+1)/\theta}b(\theta+1)}{\Gamma(1-\theta)}\,n^{-\theta}\quad\text{as }n\to\infty $$
follows by a well-known Tauberian result by Hardy, Littlewood and Karamata (see e.g.~\cite[Cor.\,1.7.3 on p.\,40]{BingGolTeug:89}). This establishes \eqref{eq:exact behavior pn} with the given constant $c$. The required monotonicity of $n(n-1)p_{n}$ will be verified next under the assumption $\rho:=a(a+b)^{-1}\le\theta (1+2\theta)^{-1}$.

\vspace{.2cm}
It follows with the help of \eqref{eq2:formula pn} that
\begin{align*}
r_{n}\ :=\ \frac{n(n+1)p_{n+1}}{n(n-1)p_{n}}\ =\ \frac{\frac{1}{n-1}\sum_{i=1}^{n}\rho^{i}c_{n+1,i}}{\sum_{i=1}^{n-1}\rho^{i}c_{n,i}}
\end{align*}
for any $n\ge 2$. Use \eqref{eq2:c_n,i} to see that
\begin{gather}
\begin{split}\label{eq:r_n estimated}
\frac{1}{n-1}\sum_{i=1}^{n}\rho^{i}c_{n+1,i}\ &=\ \sum_{i=1}^{n-1}\frac{n-1-i\theta}{n-1}\,\rho^{i}c_{n,i}\,+\,\sum_{i=2}^{n}\frac{1+i\theta}{n-1}\,\rho^{i}c_{n,i-1}\\
&=\ \sum_{i=1}^{n-1}\left(\frac{n-1-i\theta}{n-1}+\frac{1+(i+1)\theta}{n-1}\rho\right)\rho^{i}c_{n,i}\\
&=\ \sum_{i=1}^{n-1}\left(1+\frac{\rho(1+\theta)}{n-1}-\frac{(1-\rho)i\theta}{n-1}\right)\rho^{i}c_{n,i}\\
&\le\ \left(1-\frac{\theta-\rho(1+2\theta)}{n-1}\right)\sum_{i=1}^{n-1}\rho^{i}c_{n,i},
\end{split}
\shortintertext{and therefore}
r_{n}\ \le\ 1-\frac{\theta-\rho(1+2\theta)}{n-1}\nonumber
\end{gather}
for $n\ge 2$, with strict inequality for $n\ge 3$.
This shows the asserted monotonicity of the sequence $(n(n-1)p_{n})_{n\ge 2}$ when $\rho\le\theta(1+2\theta)^{-1}$.

\vspace{.2cm}
Now suppose $\rho>\theta$ for the proof of \eqref{eq:magnitude behavior pn}, which may also be stated as
$$ 0\ <\ \liminf_{n\to\infty}\frac{p_{n}}{n^{2+\theta}}\ \le\ \limsup_{n\to\infty}\frac{p_{n}}{n^{2+\theta}}\ <\ \infty. $$
Let $N$ be a random variable with law $\PF(\theta,1,\xi)$, and let $X,X_{1},X_{2},\ldots$ be iid random variables independent of $N$ and with common law $\PF(\theta,a,b)$. Then the law of $Y:=\sum_{k=1}^{N}X_{k}$ is also power-fractional, specifically
$$ Y\ \eqdist\ \PF(\theta,a,b+\xi). $$
We fix $\xi>0$ such that $a/(a+b+\xi)=\theta$, thus $b+\xi=a(1-\theta)/\theta$. Then \eqref{eq:exact behavior pn} applies to give
$$ \Prob(Y=n)\ \simeq\ \frac{a^{-1/\theta}(1-\theta)(\theta+1)}{\theta\Gamma(1-\theta)}\,n^{-(2+\theta)}\quad\text{as }n\to\infty. $$
Using $\Prob(N=1)=(1+\xi)^{-1-1/\theta}$ by \eqref{eq1:formula pn} and
$$ \Prob(Y=n)\ \ge\ \Prob(N=1,X_{1}=n)\ =\ \Prob(N=1)\,p_{n}\ =\ \frac{p_{n}}{(1+\xi)^{1+1/\theta}}, $$
we conclude
$$ \limsup_{n\to\infty}\frac{p_{n}}{n^{2+\theta}}\ \le\ \frac{a^{-1/\theta}(1-\theta)(\theta+1)(1+\xi)^{1+1/\theta}}{\theta\Gamma(1-\theta)}\ <\ \infty. $$

\vspace{.1cm}
Towards a positive lower bound for $\liminf_{n\to\infty}n^{-(2+\theta)}p_{n}$, we observe that, if $(u_{n})_{n\ge 0}=\PF(\theta,a,\zeta)$ with $\zeta>b$ is chosen such that $a/(a+\zeta)=\theta$, then, by \eqref{eq2:formula pn},
\begin{align*}
p_{n}\ =\ \frac{p_{1}}{n!}\sum_{i=1}^{n-1}\rho^{i}c_{n,i}\ \ge\ \frac{p_{1}}{n!}\sum_{i=1}^{n-1}\theta^{i}c_{n,i}\ =\ \frac{p_{1}}{u_{1}}u_{n}
\end{align*}
for all $n\ge 2$, and $p_{1}/u_{1}=(\rho/\theta)^{1+1/\theta}$ by \eqref{eq1:formula pn}. Consequently,
\begin{align*}
\liminf_{n\to\infty}\frac{p_{n}}{n^{2+\theta}}\ &\ge\ \bigg(\frac{\rho}{\theta}\bigg)^{1+1/\theta}\lim_{n\to\infty}\frac{u_{n}}{n^{2+\theta}}\\
&=\ \bigg(\frac{\rho}{\theta}\bigg)^{1+1/\theta}\frac{a^{-(\theta+1)/\theta}\zeta(\theta+1)}{\Gamma(1-\theta)}\ >\ 0,
\end{align*}
and this completes the proof of the proposition.\qed
\end{proof}

\begin{Rem}\rm
Using \eqref{eq1:c_n,i} and \eqref{eq2:c_n,i}, one can derive estimates for the $c_{n,i}$ that, with the help of \eqref{eq:r_n estimated}, show that for $\rho<\theta$, the sequence $(n(n-1)p_{n})_{n\ge 1}$ remains \emph{ultimately} nonincreasing, and thus \eqref{eq:exact behavior pn} in Theorem \ref{thm:asymptotics pn} continues to hold. However, to keep the proof concise, we have omitted the details of the~argument.
\end{Rem}


\section{Varying environment}\label{sec:2}

The composition of pgf's of power-fractional laws remains stable even when the parameters $a$ and $b$ vary, as long as $\theta$ is kept fixed. Specifically, if $N,X_{1},X_{2},\ldots$ are independent random variables such that the law of $N$ is $\PF(\theta,a_{1},b_{1})$ (with pgf~$f_{1}$) and the common law of the $X_{k}$ is $\PF(\theta,a_{2},b_{2})$ (with pgf~$f_{2}$), then the random sum $Y:=\sum_{k=1}^{N}X_{k}$ has pgf
\begin{equation}\label{eq1:iteration stable general}
f_{1}\circ f_{2}(s)\ =\ 1\,-\,\left[\frac{a_{1}a_{2}}{(1-s)^{\theta}}+a_{1}b_{2}+b_{1}\right]^{-1/\theta},\quad s\in [0,1).
\end{equation}
and thus again a law of power-fractional type, namely
\begin{gather}\label{eq2:iteration stable general}
Y\ \eqdist\ \PF(\theta,a_{1}a_{2},a_{1}b_{2}+b_{1}).
\end{gather}
This follows from \eqref{eq2:def PF gf}, but can also be deduced with the help of the functional equation \eqref{eq:general functional equation}. By fixing $\theta$ and thus $H=H_{\theta}$, we obtain
\begin{align*}
H(f_{1}\circ f_{2}(s))\ &=\ a_{1}H(f_{2}(s))+H(f_{1}(0))\\
&=\ a_{1}a_{2}H(s)+a_{1}H(f_{2}(0))+H(f_{1}(0)).
\end{align*}
Here is a mnemonic -- though somewhat unconventional in notation -- for expressing \eqref{eq1:iteration stable general} and \eqref{eq2:iteration stable general}:
\begin{gather}\label{eq3:iteration stable general}
\sum_{k=1}^{\PF(\theta,a_{1},b_{1})}\PF(\theta,a_{2},b_{2})_{k}\ =\ \PF(\theta,a_{1}a_{2},a_{1}b_{2}+b_{1}),
\end{gather}
where the independence assumptions in the original formulation are implicitly assumed.  
Generalizing to any finite number of compositions, we arrive at the following result:

\begin{Lemma}
Fixing $\theta\in (0,1]$, let $f_{k}$ for $k=1,\ldots,n$ be the pgf~of $\PF(\theta,a_{k},b_{k})$ for arbitrary $(a_{k},b_{k})$ satisfying (A1). Then $f_{1}\circ\cdots\circ f_{n}$ is the pgf of the power-frac\-tional law with parameters $\theta,\prod_{k=1}^{n}a_{k}$, and $\sum_{k=1}^{n}b_{k}\prod_{i=1}^{k-1}a_{i}$, abbreviated hereafter as
\begin{equation}\label{eq:composition rule PF}
f_{1}\circ\ldots\circ f_{n}\ \sim\ \PF\left(\theta,\prod_{k=1}^{n}a_{k},\sum_{k=1}^{n}b_{k}\prod_{i=1}^{k-1}a_{i}\right).
\end{equation}
\end{Lemma}

For linear-fractional distributions $(\theta=1)$, Eq.\,\eqref{eq:composition rule PF} has already been stated in \cite[Eq.\,(1.7)]{Alsmeyer:21} and then used to derive results for linear-fractional \GWP's in iid random environment. The latter means to interpret $(a_{n},b_{n})_{n\ge 1}$ as the outcome of a sequence $\bfe=(\bfe_{n})_{n\ge 1}$ of iid random vectors $\bfe_{n}=(A_{n},B_{n})$ (the random environment). In the following, the approach described in \cite{Alsmeyer:21} is extended to the power-fractional case. More specifically, we will study the behavior of a Galton-Watson process in random environment (\GWPRE) with iid power-fractional offspring laws $\PF(\theta,A_{n},B_{n})$ for $\theta$ fixed. As \eqref{eq:composition rule PF} shows further, the sequence describing the parameter evolution, namely
\begin{equation*}
(\Pi_{n},R_{n})\ :=\ \left(\prod_{k=1}^{n}A_{k},\sum_{k=1}^{n}B_{k}\prod_{i=1}^{k-1}A_{i}\right),\quad n\ge 0
\end{equation*}
does not depend on the choice of $\theta$, apart from the constraints imposed by (A1) on the range of the $\bfe_{n}=(A_{n},B_{n})$. It is therefore exactly the same as in \cite{Alsmeyer:21} and in fact the sequence of backward iterations of the random affine linear functions
\begin{equation}\label{eqdef:gn(x)}
g_{n}(x)\,=\,g(\bfe_{n},x)\,:=\,A_{n}x+B_{n},\quad n\ge 1.
\end{equation}
which have been extensively studied in the literature, see e.g.~\cite{Kesten:73,Vervaat:79,GolMal:00,AlsIksRoe:09} and especially the recent monographs \cite{BurDamMik:16,Iksanov:16}. A review of its most important properties for the present work will be given in the next section after some further background and notation needed throughout.

\vspace{.1cm}
As usual, let $(Z_{n})_{n\ge 0}$ denote the power-fractional \GWPRE\ to be considered hereafter, and let $(A,B)$ a generic copy of the environmental variables $(A_{n},B_{n})$ which, according to (A1), must satisfy
\begin{equation}\label{eq:parameter settings}
\Prob(A>0,\,B>0,\,A+B\ge 1)\,=\,1.
\end{equation}
Define $\bfe_{1:n}:=(\bfe_{1},\ldots,\bfe_{n})$ and $\bfe_{\ssy\geq n}:=(\bfe_{n},\bfe_{n+1},\ldots)$, so that $\bfe=\bfe_{\ssy\geq 1}$.
Given $\bfe_{n}$, let $f_{n}=f(\bfe_{n},\cdot)$ denote the \emph{quenched pgf} of the power-fractional distribution $\PF(\theta,A_{n},B_{n})=(P_{n,k})_{k\ge 0}$. From \eqref{eq1:def PF gf}, we have
$$ \frac{1}{(1-f_{n}(s))^{\theta}}\ =\ \frac{A_{n}}{(1-s)^{\theta}}\,+\,B_{n}\quad\text{a.s.} $$
Let $(P_{k})_{k\ge 0}=\PF(\theta,A,B)$ be a generic copy of the $(P_{n,k})_{k\ge 0}$. 

\vspace{.1cm}
Now, conditioned $\bfe_{1:n}$, the individuals of generation $n-1$ produce offspring according to the power-fractional distribution $\PF(\theta,\gamma,A_{n},B_{n})$, which has the pgf $f_{n}$. Therefore 
$$ f_{1:n}\ :=\ f_{1}\circ\ldots\circ f_{n} $$
equals the random pgf of (the quenched law of) $Z_{n}$ given both $\bfe_{1:n}$ and also the full environment $\bfe$. Equivalently, by \eqref{eq:composition rule PF},
\begin{gather}\label{eq:quenched law Z_n}
\Prob(Z_{n}\in\cdot|\bfe_{1:n})\ =\ \Prob(Z_{n}\in\cdot|\bfe)\ =\ \PF(\theta,\Pi_{n},R_{n})
\shortintertext{and}
f_{1:n}(s) =\ 1\ -\ \left[\frac{\Pi_{n}}{(1-s)^{\theta}}\,+\,R_{n}\right]^{-1/\theta}\label{eq:quenched gf Z_n}
\end{gather}
for each $n\in\N$, where
$$ \Pi_{n}:=\prod_{k=1}^{n}A_{k}\quad\text{and}\quad R_{n}:=\sum_{k=1}^{n}\Pi_{k-1}B_{k}. $$ 
Since, with $g_{n}$ defined in \eqref{eqdef:gn(x)} and $\psi(s):=(1-s)^{-\theta}$,
\begin{equation}\label{eq:basic identity for f_1:n}
\psi\circ f_{1:n}(s)\ =\ \frac{1}{(1-f_{1:n}(s))^{\theta}}\ =\ \frac{\Pi_{n}}{(1-s)^{\theta}}\,+\,R_{n}\ =\ g_{1:n}\circ\psi(s)
\end{equation}
for $s\in [0,1)$, we see that, up to conjugation, $(f_{1:n})_{n\ge 0}$ equals the sequence of backward iterations of the iid~random affine linear maps $g_{1},g_{2},\ldots$ already mentioned. The corresponding forward iterations $g_{n:1}(x):=g_{n}\circ\ldots\circ g_{1}(x)$ form a Markov chain on $[0,\infty)$ with initial state $x$ and also a so-called \emph{iterated function system (IFS)}. Based on these observations and similar to \cite{Alsmeyer:21} for the linear-fractional case, we are able to derive properties of \GWPRE's with power-fractional offspring laws by drawing on results about iterations of the $g_{n}$. For instance, when defining 
$$ q_{n}(\bfe_{1:n})\,:=\,\Prob(Z_{n}=0|\bfe_{1:n})\,=\,f_{1:n}(0), $$
we infer as an immediate consequence of \eqref{eq:basic identity for f_1:n} that
\begin{equation}\label{eq:survival probab quenched}
q_{n}(\bfe_{1:n})\ =\ 1\,-\,\big(\Pi_{n}+R_{n}\big)^{-1/\theta}\quad\text{a.s.}
\end{equation}
for all $n\ge 1$. Let
$$ q(\bfe)\ :=\ \lim_{n\to\infty}q_{n}(\bfe_{1:n}) $$
denote its a.s.~limit, which is the quenched extinction probability of $(Z_{n})_{n\ge 0}$ given $\bfe$.

\vspace{.2cm}
\emph{Reversing the environment}. As in \cite{Alsmeyer:21}, it will be preferable for the presentation of some of our results to consider the given \GWPRE\ up to time $n$ under the time-reversed environment $\bfe_{n:1}=(\bfe_{n},\ldots,\bfe_{1})$ rather than $\bfe_{1:n}$ which means to use the random offspring laws $\PF(\theta,A_{k},B_{k})$ in reverse order. This does not change the (annealed) law of $Z_{0:n}:=(Z_{0},\ldots,Z_{n})$, but the quenched laws are naturally different. On the other hand, they share the same distribution as random measures, see \eqref{eq:equal quenched laws} below, and this will sometimes be used hereafter to formulate assertions about quenched asymptotic behavior in a more tangible form than they would appear without time-reversal. To be more specific, we define
\begin{gather}
\bfP\,:=\,\Prob(\cdot|\bfe),\quad\bfP^{(1:n)}\,:=\,\Prob(\cdot|\bfe_{1:n})\quad\text{and}\quad\bfP^{(n:1)}\,:=\,\Prob(\cdot|\bfe_{n:1}),\nonumber
\intertext{with corresponding expectations $\bfE,\,\bfE^{(1:n)}$ and $\bfE^{(n:1)}$, and note that}
\begin{split}
\bfP(Z_{0:n}\in\cdot)\ &=\ \bfP^{(1:n)}(Z_{0:n}\in\cdot)\ \eqdist\ \bfP^{(n:1)}(Z_{0:n}\in\cdot)
\end{split}
\label{eq:equal quenched laws}
\end{gather}
for each $n\in\N$. Then $\bfP^{(n:1)}(Z_{n}\in\cdot)$ has pgf~$f_{n:1}:=f_{n}\circ\ldots\circ f_{1}$, giving
\begin{align}
\bfP^{(n:1)}(Z_{n}\in\cdot)\ =\ \PF\left(\theta,\Pi_{n},\Pi_{n}\sum_{k=1}^{n}\Pi_{k}^{-1}B_{k}\right)
\label{eq:quenched law Z_n bw}
\end{align}
upon using \eqref{eq:composition rule PF} for $f_{n:1}$. We thus see that, when opposed to $\bfP^{(1:n)}(Z_{n}\in\cdot)$, the random parameter $R_{n}$ is exchanged by $\Pi_{n}R_{n}^{(-1)}$, where
\begin{equation}\label{eq:def Rn dual}
R_{n}^{(-1)}\,:=\,\sum_{k=1}^{n}\Pi_{k}^{-1}B_{k}
\end{equation}
and $(\Pi_{n},R_{n})\eqdist (\Pi_{n},R_{n}^{(-1)})$ holds for all $n$ by \eqref{eq:equal quenched laws}.

\vspace{.2cm}
\emph{Classification}. For a \GWPRE, the distinction between subcritical, critical and supercritical type is based on the behavior of the logarithm of the quenched mean $\log\bfE Z_{n}$ as $n\to\infty$. Provided this quantity is a.s.~finite for all $n$, it constitutes an ordinary random walk with generic increment $\log f'(1)$ which in the present situation equals $-\theta^{-1}\log A$.
In fact, if $Z_{0}=1$, then
$$ \theta\log\bfE Z_{n}\ =\ -\log\Pi_{n}\ =\ -\sum_{k=1}^{n}\log A_{k}\ =:\ S_{n}\quad\text{a.s.} $$
for all $n\ge 0$. Depending on the fluctuation-type of the walk $(S_{n})_{n\ge 0}$ (the factor $\theta^{-1}$ does not matter), namely
\begin{itemize}\itemsep2pt
\item[$\bullet$] positive divergence, i.e., $S_{n}\to\infty\text{ a.s.}$,
\item[$\bullet$] negative divergence, i.e., $S_{n}\to-\infty\text{ a.s.}$,
\item[$\bullet$] oscillation, i.e. $\textstyle\limsup_{n\to\infty}S_{n}=+\infty\text{ and }\liminf_{n\to\infty}S_{n}=-\infty\text{ a.s.}$,
\item[$\bullet$] degeneracy, i.e., $S_{n}=0\text{ a.s~for all }n\ge 0$,
\end{itemize}
the process $(Z_{n})_{n\ge 0}$ is called subcritical, supercritical, critical or strongly critical, respectively \cite[Def.~2.3]{KerstingVatutin:17}. If $\Erw\log A$ exists, this means that
\begin{align*}
(Z_{n})_{n\ge 0}\text{ is }
\begin{cases}
\hfill\text{subcritical}&\text{if }\Erw\log A>0,\\
\hfill\text{critical}&\text{if }\Erw\log A=0\text{ and }\Prob(A\ne 1)>0,\\
\text{strongly critical}&\text{if }A=1\text{ a.s.},\\
\hfill\text{supercritical}&\text{if }\Erw\log A<0.
\end{cases}
\end{align*}
Based on the very explicit knowledge of the $f_{1:n}$ in the given situation, a precise description of when each of these cases occurs is possible and was given in some detail in \cite[Sect.\,2]{Alsmeyer:21}. In the following section, we therefore limit ourselves to a summary of the most important facts, in particular with regard to the asymptotic behavior of the variables $\Pi_{n},R_{n}$ and $R_{n}^{(-1)}$, which play a role in the development of the random parameters of the deleted laws of $Z_{n}$.


\section{Prerequisites about iterations of random affine linear functions}\label{sec:random affine recursions}

Recall that $g_{n}(x)=A_{n}x+B_{n}$ and Eq.\,\eqref{eq:basic identity for f_1:n} for the backward iterations $g_{1:n}$ of the IFS and Markov chain  $(g_{n:1}(x))_{n\ge 0}$ generated by the $g_{n}$. The latter shows that
\begin{equation*}
g_{1:n}(x)\ =\ \psi\circ f_{1:n}\circ\psi^{-1}(x)\ =\ \Pi_{n}x\,+\,R_{n},\quad n\ge 1,
\end{equation*}
for any $n\ge 1$, which does not depend on $\theta$. It also explains the appearance of the $(\Pi_{n},R_{n})$ mentioned above in connection with the quenched laws of the given power-fractional \GWPRE. Goldie and Maller \cite[Thm.~2.1]{GolMal:00} have provided necessary and sufficient conditions for the stability (positive recurrence) of $(g_{n:1}(x))_{n\ge 0}$ that goes along with the almost sure convergence of the $g_{1:n}(x)$ as $n\to\infty$. Before  summarizing below their result for the case of positive $A,B$ satisfying an additional nondegeneracy condition, we further introduce a second sequence of random affine linear maps that is of interest here and in dual relation to the $g_{n}$. Indeed, defining
\begin{gather*}
g_{n}^{(-1)}(x)\,:=\,A_{n}^{-1}x+A_{n}^{-1}B_{n}\quad\text{for }n\in\N,
\intertext{the associated backward iterations take the form}
g_{1:n}^{(-1)}\,=\,\Pi_{n}^{-1}x\,+\,R_{n}^{(-1)},\quad n\in\N,
\intertext{with $R_{n}^{(-1)}$ defined by \eqref{eq:def Rn dual}, and the duality relation}
\frac{R_{n}}{\Pi_{n}}\ =\ \frac{g_{1:n}(0)}{\Pi_{n}}\ =\ g_{n:1}^{(-1)}(0)\ \eqdist\ g_{1:n}^{(-1)}(0)\ =\ R_{n}^{(-1)}
\end{gather*}
holds for all $n\in\N$.

\begin{Prop}\label{prop:GoldieMaller}
Suppose that $A,B$ are a.s. positive and
\begin{equation}\label{eq:nondegeneracy}
\Prob(Ax+B=x)\,<\,1\quad\text{for all }x\in\R.
\end{equation}
Define $J^{\pm}(x):=\Erw(x\wedge\log^{\pm}A)$, and let
\begin{equation*}
R_{\infty}\,:=\,\sum_{k\ge 1}\Pi_{k-1}B_{k}\quad\text{and}\quad R_{\infty}^{(-1)}\,:=\,\sum_{k\ge 1}\Pi_{k}^{-1}B_{k},
\end{equation*}
denote the increasing limits of $R_{n}$ and $R_{n}^{(-1)}$, respectively, called perpetuities. Then the following assertions hold:
\begin{itemize}\itemsep3pt
\item[(a)] The condition
\end{itemize}
\begin{equation}\label{eq:GolMal cond}
\Pi_{n}\,\to\,0\text{ a.s.}\quad\text{and}\quad
I_{-}\,:=\,\int_{[1,\infty)}\frac{\log x}{J^{-}(\log x)}\,\Prob(B\in dx)\,<\,\infty
\end{equation}
\begin{itemize}\itemsep3pt
\item[] is necessary and sufficient for $R_{\infty}<\infty$, the almost sure convergence of $g_{1:n}(x)$ to $R_{\infty}$, and the convergence in law of the forward iterations $g_{n:1}(x)$ to the same limit.
\item[(b)] The condition
\end{itemize}
\begin{equation}\label{eq:GolMal cond dual}
\Pi_{n}\,\to\,\infty\text{ a.s.}\quad\text{and}\quad
I_{+}\,:=\,\int_{[1,\infty)}\frac{\log x}{J^{+}(\log x)}\,\Prob\bigg(\frac{B}{A}\in dx\bigg)\,<\,\infty
\end{equation}
\begin{itemize}\itemsep3pt
\item[] is necessary and sufficient for $R_{\infty}^{(-1)}<\infty$, the almost sure convergence of $g_{1:n}^{(-1)}(x)$ to $R_{\infty}^{(-1)}$, and the convergence in law of the forward iterations $g_{n:1}(x)$ to the same limit.
\end{itemize}
In the case when \eqref{eq:nondegeneracy} fails, that is, $Ax+B=x$ a.s.~for some $x$, the violation of \eqref{eq:GolMal cond}, respectively \eqref{eq:GolMal cond dual} entails $R_{\infty}=\infty$ a.s., respectively $R_{\infty}^{(-1)}=\infty$ a.s.
\end{Prop}

We note that we always have $R_{\infty}\ge 1$ a.s.~as a consequence of the basic constraint \eqref{eq:parameter settings} on $(A,B)$ (see \cite[Eq.\,(16)]{Alsmeyer:21}). It is also easily verified that $R_{\infty}$ and $R_{\infty}^{(-1)}$ cannot be a.s. finite at the same time. Therefore, the trichotomy
\begin{itemize}\itemsep=2pt
\item[(C1)] $R_{\infty}<\infty=R_{\infty}^{(-1)}$ a.s.
\item[(C2)] $R_{\infty}^{(-1)}<\infty=R_{\infty}$ a.s.
\item[(C3)] $R_{\infty}=R_{\infty}^{(-1)}=\infty$ a.s.
\end{itemize}
holds. A complete characterization of the three cases in terms of $(A,B)$ is possible with the help of the previous proposition and has been provided in \cite{Alsmeyer:21}. We refrain from repeating it here and only note that, as an outcome of this characterization, the \GWPRE\ $(Z_{n})_{n\ge 0}$ is
\begin{itemize}\itemsep2pt
\item[] supercritical if (C1) holds;
\item[] subcritical if (C2) holds;
\item[] critical if (C3) holds and $\Prob(A\ne 1)>0$;
\item[] strongly critical if (C3) holds and $A=1$ a.s..
\end{itemize}

\section{Basic results for power-fractional \GWP's in a constant environment}\label{sec:fixed env}

Before proceeding with the random environment case, we devote the present section to a survey of the most basic results for power-fractional \GWP's in a constant environment, which in view of our standing assumption means that the offspring law is $\PF(\theta,a,b)$ for some $\theta\in (0,1]$ and $a,b>0$ such that $a+b\ge 1$. 

\vspace{.2cm}\noindent
\textbf{\emph{1. Supercritical case.}} Recall from the discussion in Section \ref{sec:1} that this case occurs iff $a\in (0,1)$ in which case the mean offspring equals $\sfm=a^{-1/\theta}$ (see also Table \ref{tab:1.3}). The two basic results, collected in the subsequent theorem, are about the extinction probability $q$ and the limiting behavior of the normalization of the process which is well-known to be a nonnegative martingale.

\begin{Theorem}\label{thm:supercritical fixed env}
Let $(Z_{n})_{n\ge 0}$ be a supercritical power-fractional \GWP\ with one ancestor, offspring distribution $\PF(\theta,a,b)$ and normalization $W_{n}=\sfm^{-n}Z_{n}$ for $n\ge 0$. Denote by $f$ the offspring pgf Then the following assertions hold:
\begin{itemize}\itemsep2pt
\item[(a)] The extinction probability $q$ is given by
\begin{equation}\label{eq:ext probab fixed environment}
q\ =\ 1\,-\,\bigg(\frac{1-a}{b}\bigg)^{1/\theta}.
\end{equation}
\item[(b)] As $n\to\infty$, the normalization $W_{n}$ converges a.s.~and in $L^{1}$ to a random variable $W_{\infty}$ with Laplace transform
\begin{equation}\label{eq:LT of W}
\vph(u)\ =\ \Erw e^{-uW_{\infty}}\ =\ 1\,-\,\left[\frac{1}{u^{\theta}}+\frac{b}{1-a}\right]^{-1/\theta},\quad u\ge 0,
\end{equation}
which in turn satisfies the functional equation $\vph(u)=f\circ\vph(a^{1/\theta}u)$ for all $u\ge 0$, known as (the multiplicative form of) Abel's equation.
\end{itemize}
\end{Theorem}

\begin{proof}
(a) Recalling that $b_{n}=b\sum_{k=0}^{n-1}a^{k}=b(1-a)^{-1}(1-a^{n})$, \eqref{eq:ext probab fixed environment} follows from
\begin{align*}
q\ =\ \lim_{n\to\infty}\Prob(Z_{n}=0)\ =\ \lim_{n\to\infty}f^{n}(0)\ =\ \lim_{n\to\infty}1\,-\,[a^{n}+b_{n}]^{-1/\theta}.
\end{align*}

(b) Since $\Erw e^{-uW_{n}}=\Erw e^{-uZ_{n}/\sfm^{n}}=f^{n}(e^{-u/\sfm^{n}})$ and $\sfm^{n}=a^{-n/\theta}$, we obtain by another use of \eqref{eq:ext probab fixed environment} that
\begin{align*}
\frac{1}{(1-\Erw e^{-uW_{n}})^{\theta}}\ &=\ \frac{1}{(1-f^{n}(e^{-u/\sfm^{n}}))^{\theta}}\ =\ \frac{a^{n}}{(1-e^{-u/\sfm^{n}})^{\theta}}\,+\,b_{n}\\
&=\ \left[\frac{a^{n/\theta}}{1-e^{-ua^{n/\theta}}}\right]^{\theta}\,+\,b_{n}.
\end{align*}
Now the last expression is easily seen to converge to $u^{-\theta}+\frac{b}{1-a}$ as $n\to\infty$, and this proves \eqref{eq:LT of W}. Moreover,
$$ \Prob(W_{\infty}=0)\,=\,\lim_{u\to\infty}\vph(u)\,=\,1\,-\,\bigg(\frac{1-a}{b}\bigg)^{1/\theta}\,=\,q $$
so that $W_{n}$ must also converge in $L^{1}$ to $W_{\infty}$ by the Kesten-Stigum theorem.
The asserted functional equation is a well-known fact that can be found, e.g., in \cite[Eq.\,(5) on p.\,10]{Athreya+Ney:72}.\qed
\end{proof}

In the linear-fractional case $\theta=1$, we obtain $q=\frac{a+b-1}{b}$ and then
\begin{equation}\label{eq:LT Winfty linear-fractional case}
\vph(u)\ =\ q\,+\,(1-q)\,\frac{1-q}{1-q+u}.
\end{equation}
This shows that the law of the normalized limit $W_{\infty}$ equals a mixture of the Dirac measure at 0 and the exponential law with parameter $1-q$, here called \emph{continuous linear-fractional law} and abbreviated as $\CLF$. Its precise definition will be given below within a wider class of distributions.

\vspace{.1cm}
For general $\theta\in (0,1]$ and with $q$ given by \eqref{eq:ext probab fixed environment}, $\vph(u)$ can be rewritten as
\begin{equation}\label{eq:LT Winfty power-fractional case}
\vph(u)\ =\ q\,+\,(1-q)\left[1-\left(1-\frac{(1-q)^{\theta}}{u^{\theta}+(1-q)^{\theta}}\right)^{1/\theta}\right]
\end{equation}
Thus the law of $W_{\infty}$ is also a mixture of the Dirac measure at 0 with a law that has Laplace transform appearing in square brackets in \eqref{eq:LT Winfty power-fractional case}. Both the latter law and the mixture are referred to as the \emph{continuous power-fractional law} and abbreviated as $\CPF$. To refine the parametrization, we return to Eq.\,\eqref{eq:LT of W} and rewrite it as:
\begin{align*}
\frac{1}{(1-\vph(u))^{\theta}}\ =\ \frac{1}{u^{\theta}}\,+\,\frac{b}{1-a},
\end{align*}
thus in a form very similar to Eq.\,\eqref{eq2:def PF gf} for power-fractional laws. This suggests to define $\CPF(\theta,\alpha,\beta)$ for $\theta\in (0,1],\,\alpha>0$ and $\beta\ge 1$ to be the law with Laplace transform $\vph(u)=1-\big[\alpha u^{-\theta}+\beta\big]^{-1/\theta}$, giving
\begin{gather}
\frac{1}{(1-\vph(u))^{\theta}}\ =\ \frac{\alpha}{u^{\theta}}\,+\,\beta.
\label{eq:CPF functional equation}
\shortintertext{and}
\vph(u)\ =\ (1-\beta^{-1/\theta})\,+\,\beta^{-1/\theta}\left[1-\left(1-\frac{\alpha/\beta}{u^{\theta}+\alpha/\beta}\right)^{1/\theta}\right].\label{eq:CPF mixing form LT}
\end{gather}
If we now let $\CPF_{+}(\theta,\alpha):=\CPF(\theta,\alpha,1)$ be the continuous power-fractional law on the \emph{positive} halfline with Laplace transform
\begin{gather}
\vph_{\theta,\alpha}(u)\,:=\,1-\left(1-\frac{\alpha}{u^{\theta}+\alpha}\right)^{1/\theta}\ =\ 1-\left(\frac{u^{\theta}}{u^{\theta}+\alpha}\right)^{1/\theta},
\label{eq:CPFplus LT}
\intertext{then \eqref{eq:CPF mixing form LT} means that}
\CPF(\theta,\alpha,\beta)\,:=\,(1-\beta^{-1/\theta})\delta_{0}\,+\,\beta^{-1/\theta}\,\CPF_{+}(\theta,\alpha/\beta).\label{eq:CPF mixing form}
\end{gather}
The continuous linear-fractional laws occur if $\theta=1$, and we therefore define
\begin{equation}\label{eq:def CLF} 
\CLF(\alpha,\beta)\,:=\,\CPF(1,\alpha,\beta)\ =\ (1-\beta)\delta_{0}\,+\,\beta\,\textit{Exp}(\alpha).
\end{equation}
Finally, \eqref{eq:LT Winfty power-fractional case} can now be restated as
\begin{equation}\label{eq:kaw of Winfty general}
W_{\infty}\,\eqdist\,\CPF(\theta,1,(1-q)^{-\theta}).
\end{equation}
\indent We finish our discussion about the law of $W_{\infty}$ with two lemmata. The first one provides an extension of the well-known fact that a geometric sum of iid exponentials is again an exponential random variable and, more generally, a linear-fractional sum of continuous linear-fractionals is again of the latter type.

\begin{Lemma}\label{lem:random sum law CPF}
Given $\theta\in (0,1]$ and positive $\alpha,\beta,a,b$ with $\beta\ge 1$ and $b\ge 1-a$, let $N,Y_{1},Y_{2},\ldots$ be independent random variables such that
\begin{gather}
N\,\eqdist\,\PF(\theta,1,a,b)\quad\text{and}\quad Y_{n}\,\eqdist\,\CPF(\theta,\alpha,\beta)\quad\text{for each }n.\nonumber
\shortintertext{Then}
\label{eq:random sum law CPF}
\sum_{k=1}^{N}Y_{k}\ \eqdist\ \CPF(\theta,a\alpha,a\beta+b)
\end{gather}
\end{Lemma}

\begin{proof}
With $f$ and $\vph$ denoting the pgf~of $N$ and the Laplace transform of the $Y_{n}$, respectively, the random sum in \eqref{eq:random sum law CPF} has Laplace transform $f\circ\vph$. A combination of the functional equations \eqref{eq2:def PF gf} and \eqref{eq:CPF mixing form LT} now yields the assertion.\qed
\end{proof}

The second lemma reveals a connection between continuous power-fract\-ional, Sibuya and Mittag-Leffler distributions of order $\theta$. The latter have the Laplace transform $\psi_{\theta}(u)=(1+u^{\theta})^{-1}$ for some $\theta\in (0,1]$ and are abbreviated as $\ML(\theta)$ here. When $\theta=1$, this reduces to the standard exponential distribution.

\begin{Lemma}
Let $\theta\in (0,1]$ and $S,X,Y_{1},Y_{2},\ldots$ be independent random variables such that $S\eqdist\Sib(\theta)$, $X\eqdist\ML(\theta)$, and the $Y_{n}$ have common law $\CPF_{+}(\theta,1)$. Then
\begin{equation}
X\ \eqdist\ \sum_{n=1}^{S}Y_{n},
\end{equation}
that is, the Mittag-Leffler distribution can be represented as a Sibuya sum of  continuous power-fractional variables on $(0,\infty)$.
\end{Lemma}

\begin{proof}
This follows again by a look at Laplace transforms. The simple details can be omitted.\qed
\end{proof}

\vspace{.2cm}\noindent
\textbf{\emph{2. Subcritical case.}} This case arises when $a>1$, and the following theorem presents the well-known results regarding the behavior of the survival probability and the asymptotic law of the process conditioned on survival, often referred to as the Yaglom law.

\begin{Theorem}\label{thm:subcritical fixed env}
Let $(Z_{n})_{n\ge 0}$ be a subcritical power-fractional \GWP\ with one ancestor, offspring distribution $\PF(\theta,a,b)$ and offspring pgf~$f$. Then the following assertions hold:
\begin{itemize}\itemsep2pt
\item[(a)] The survival probability $\Prob(Z_{n}>0)$ satisfies
\begin{equation}\label{eq:surv probab fixed environment}
\lim_{n\to\infty}a^{n/\theta}\,\Prob(Z_{n}>0)\ =\ \left(\frac{a-1}{a+b-1}\right)^{1/\theta}.
\end{equation}
\item[(b)] The quasi-stationary limit law of the process when conditioned upon survival is again power-fractional, namely
\begin{equation}\label{eq:Yaglom law fixed environment}
\lim_{n\to\infty}\Prob(Z_{n}\in\cdot|Z_{n}>0)\ =\ \PF\left(\theta,\frac{a-1}{a+b-1},\frac{b}{a+b-1}\right)
\end{equation}
in the sense of total variation convergence.
\end{itemize}
\end{Theorem}

\begin{proof}
(a) With $b_{n}$ as in \eqref{eq:def LF gf} and $Z_{n}\eqdist\PF(\theta,a^{n},b_{n})$, we obtain
\begin{align*}
a^{n/\theta}\,\Prob(Z_{n}>0)\ &=\ a^{n/\theta}\,(1-f^{n}(0))\ =\ \left(\frac{a^{n}}{a^{n}+b_{n}}\right)^{1/\theta}\\
&=\ \left(1+b\sum_{k=0}^{n-1}a^{k-n}\right)^{-1/\theta}\ =\ \left(1+\frac{b(1-a^{-n})}{a-1}\right)^{-1/\theta}
\end{align*}
and thereby \eqref{eq:surv probab fixed environment} upon letting $n\to\infty$.

\vspace{.1cm}
(b) Here we invoke Lemma \ref{lem:conditional law PF} to infer
$$ \Prob(Z_{n}\in\cdot|Z_{n}>0)\ =\ \PF\left(\theta,\frac{a^{n}}{a^{n}+b_{n}},\frac{b_{n}}{a^{n}+b_{n}}\right), $$
and we have just shown for (a) that $a^{n}/(a^{n}+b_{n})\to (a-1)/(a+b-1)$. The asserted convergence is now immediate.\qed
\end{proof}

Note that the above formulae for the survival probability $\Prob(Z_{n}>0)$ and the conditional law $\Prob(Z_{n}\in\cdot|Z_{n}>0)$ hold regardless of the value of $a$.

\vspace{.2cm}\noindent
\textbf{\emph{2. Critical case.}} Finally turning to the critical case when $a=1$, we recall from Table \ref{tab:1.3} that the offspring variance is infinite unless the law is linear-fractional and thus $\theta=1$. By drawing on the proof of Theorem \ref{thm:subcritical fixed env}, but with $a=1$ and thus $b_{n}=bn$, we obtain:

\begin{Theorem}\label{thm:critical fixed env}
Let $(Z_{n})_{n\ge 0}$ be a critical power-fractional \GWP\ with one ancestor, offspring distribution $\PF(\theta,1,b)$ and offspring pgf~$f$. Then the following assertions hold:
\begin{itemize}\itemsep2pt
\item[(a)] The survival probability $\Prob(Z_{n}>0)$ satisfies
\begin{gather}\label{eq:surv probab fixed environment}
\lim_{n\to\infty}n^{1/\theta}\,\Prob(Z_{n}>0)\ =\ b^{-1/\theta}.
\shortintertext{Furthermore,}
\lim_{n\to\infty}n^{-1/\theta}\,\Erw(Z_{n}|Z_{n}>0)\ =\ b^{1/\theta}.\label{eq:cond mean fixed environment}
\end{gather}
\item[(b)] The conditional law of $(bn)^{-1/\theta}Z_{n}$ given $Z_{n}>0$ converges weakly as $n\to\infty$, namely
\begin{equation}\label{eq:Kolmogorov-Yaglom limit fixed env}
\Prob\left(\frac{Z_{n}}{(bn)^{1/\theta}}\in\cdot\bigg|Z_{n}>0\right)\ \weakly\ \CPF_{+}(\theta,1).
\end{equation}
\end{itemize}
\end{Theorem}

In the linear-fractional case $\theta=1$, Eq.\,\eqref{eq:Kolmogorov-Yaglom limit fixed env} reduces to the familiar Kolmo\-gorov-Yaglom exponential limit theorem after pointing out that $\CPF_{+}(1,1)$ equals the standard exponential distribution and $2b=\Var Z_{1}$.

\begin{proof}
(a) By drawing on the formulae pointed out in the proof of Theorem \ref{thm:subcritical fixed env}, but now with $a=1$ and $b_{n}=bn$, we obtain
$$ n^{1/\theta}\,\Prob(Z_{n}>0)\ =\ \frac{n^{1/\theta}}{\Erw(Z_{n}|Z_{n}>0)}\ =\ \left(\frac{n}{1+bn}\right)^{1/\theta} $$
and then obviously the asserted limits in \eqref{eq:surv probab fixed environment} and \eqref{eq:cond mean fixed environment}.

\vspace{.1cm}
(b) Put $Y_{n}=(bn)^{-1/\theta}Z_{n}$ and $h_{n}(s)=\Erw(s^{Z_{n}}|Z_{n}>0)$ for $n\in\N_{0}$.
Since $\Prob(Z_{n}\in\cdot|Z_{n}>0)=\PF(\theta,(1+bn)^{-1},1-(1+bn)^{-1})$, we infer
\begin{gather*}
\begin{split}
\frac{1}{(1-\Erw(e^{-uY_{n}}|Z_{n}>0))^{\theta}}\ &=\ \frac{1}{(1-h_{n}(e^{-u/(bn)^{1/\theta}})^{\theta}}\\
&=\ \frac{(1+bn)^{-1}}{(1-e^{-u/(bn)^{1/\theta}})^{\theta}}\,+\,\frac{bn}{1+bn}
\end{split}
\shortintertext{or, equivalently,}
\Erw(e^{-uY_{n}}|Z_{n}>0)\ =\ 1\,-\,\left[\frac{(1+bn)^{-1}}{(1-e^{-u/(bn)^{1/\theta}})^{\theta}}\,+\,\frac{bn}{1+bn}\right]^{-1/\theta}
\end{gather*}
for any $u\ge 0$. But the last expression converges to $1-[u^{-\theta}+1]^{-1/\theta}$ which, by \eqref{eq:CPFplus LT}, equals the Laplace transform of $\CPF_{+}(\theta,1)$.\qed
\end{proof}


\section{Basic results for power-fractional \GWP's in random environment}

In the following, let $(Z_{n})_{n\ge 0}$ be a power-fractional GWP\ with one ancestor in the iid random environment $\bfe=(A_{n},B_{n})_{n\ge 0}$ as described earlier, i.e., for individuals in generation $n-1$ the conditional offspring law given $\bfe$ is $\PF(\theta,A_{n},B_{n})$ with pgf~$f_{n}=f(A_{n},B_{n},\cdot)$ for each $n\ge 1$, where $\theta\in (0,1]$ and
\begin{equation}\label{eq:basic parameter condition}
\Prob(A_{n}>0,B_{n}>0\text{ and }A_{n}+B_{n}\ge 1)\,=\,1.
\end{equation}
The next three theorems can be viewed as the counterparts to those in Section for fixed environment.

\begin{Theorem}\label{thm:supercritical iid env}
Let $(Z_{n})_{n\ge 0}$ be supercritical and thus $R_{\infty}<\infty=R_{\infty}^{(-1)}$ a.s. Then the following assertions hold:
\begin{itemize}\itemsep2pt
\item[(a)] The quenched extinction probability $q(\bfe)$ is given by
\begin{equation}\label{eq:ext probab varying environment}
q(\bfe)\ =\ 1\,-\,R_{\infty}^{-1/\theta}\quad\text{a.s.}
\end{equation}
and $\Prob(q(\bfe)<1)=1$. Moreover,
\begin{equation}\label{eq:functional eq q(bfe)}
\frac{1}{(1-q(\bfe))^{\theta}}\ =\ \frac{A_{1}}{(1-q(\bfe_{\ssy\geq 2}))^{\theta}}\,+\,B_{1}\quad\text{a.s.}
\end{equation}
\item[(b)] For almost all realizations of $\bfe$, the sequence $W_{n}=\Pi_{n}^{1/\theta}Z_{n}$, $n\ge 0$,  forms a nonnegative martingale with mean 1 under the quenched probability measure $\bfP$ and, as $n\to\infty$, it converges a.s.~and in $L^{1}$ to a random variable $W_{\infty}$ with quenched law $\CPF(\theta,1,R_{\infty})$ having Laplace transform
\begin{gather}\label{eq:LT of W VE}
\vph(\bfe,u)\ =\ \bfE e^{-uW_{\infty}}\ =\ 1\,-\,\left[\frac{1}{u^{\theta}}+R_{\infty}\right]^{-1/\theta},\quad u\ge 0,
\shortintertext{which in turn satisfies}
\vph(\bfe,u)\,=\,f_{1}\circ\vph\big(\bfe_{\ssy\geq 2},A_{1}^{1/\theta}u\big)\quad\text{a.s.}
\end{gather}
for all $u\ge 0$ and $n\in\N$.
\end{itemize}
\end{Theorem}

\begin{proof}
(a) Going back to Athreya and Karlin \cite{AthreyaKarlin:71a}, it is known that
\begin{equation}\label{eq2:functional eq q(bfe)}
q(\bfe_{\geq 1})\ =\ f_{1}(q(\bfe_{\geq 2}))\quad\text{a.s.}
\end{equation}
and that $\{q(\bfe_{\geq 1})=1\}$ is a.s. shift-invariant, i.e.
$$ \{q(\bfe_{\geq 1})=1\}=\{q(\bfe_{\geq 2})=1\}\quad\text{a.s.} $$
By ergodicity of $\bfe$, it has probability 0 or 1, and since, by \eqref{eq:survival probab quenched},
\begin{equation*}
q(\bfe)\ =\ \lim_{n\to\infty}q_{n}(\bfe_{1:n})\ =\ 1\,-\,\lim_{n\to\infty}(\Pi_{n}+R_{n})^{-\theta}\ =\ 1\,-\,R_{\infty}^{-\theta}\quad\text{a.s.}
\end{equation*}
under the given assumptions, we see that \eqref{eq:ext probab varying environment} and in particular $0<q(\bfe)<1$ a.s. holds. The latter allows us to rewrite \eqref{eq2:functional eq q(bfe)} in the asserted form \eqref{eq:functional eq q(bfe)}.

\vspace{.1cm}
(b) Since $\Pi_{n}^{-1/\theta}$ equals the quenched mean of $Z_{n}$, it is a well-known fact that $(W_{n})_{n\ge 0}$ as the quenched normalization sequence constitutes a $\bfP$-martingale. To derive the Laplace transform of its a.s.~and $L^{1}$-limit $W_{\infty}$, one can proceed in a similar manner as in the proof of Theorem \ref{thm:supercritical fixed env}(b), namely
\begin{align*}
\frac{1}{(1-\bfE e^{-uW_{n}})^{\theta}}\ &=\ \frac{1}{\big(1-f_{1:n}(e^{-u\Pi_{n}^{1/\theta}})\big)^{\theta}}\ =\ \left[\frac{\Pi_{n}^{1/\theta}}{1-e^{-u\Pi_{n}^{1/\theta}}}\right]^{\theta}\,+\,R_{n}\\
&\xrightarrow{n\to\infty}\ \frac{1}{u^{\theta}}\,+\,R_{\infty}\quad\text{a.s.}
\end{align*}
for any $u\ge 0$. All remaining assertions are now easily verified.\qed
\end{proof}

Let us write $\PF_{+}(\theta,a)$ as shorthand for $\PF(\theta,a,1-a)$, the power-fractional law on the positive integers $\N$ with parameters $\theta\in (0,1]$ and $a\in (0,1)$. In the case $\theta=1$, this law equals the familiar geometric law on $\N$, denoted as $\Geom_{+}(a)$. For the two next results about the subcritical and the critical case we point out beforehand that, by Lemma \ref{lem:conditional law PF} and \eqref{eq:quenched law Z_n}, the conditional law under $\bfP$ of $Z_{n}$ given $Z_{n}>0$ equals a.s.
\begin{equation}
\bfP(Z_{n}\in\cdot|Z_{n}>0)\ =\ \PF_{+}\left(\theta,\frac{1}{1+R_{n}/\Pi_{n}}\right)
\end{equation}
for each $n\in\N$, and under the reversed environment $\bfe_{n:1}$, i.e., under $\bfP^{(n:1)}$ instead of $\bfP^{(1:n)}$ (or $\bfP$, which makes no difference),
the same result holds with $R_{n}^{(-1)}$ in the place of $R_{n}/\Pi_{n}$.
Let $h_{n}(s)=h(\bfe_{1:n},s)$ and $h_{n}^{*}(s)=h(\bfe_{n:1},s)$ denote the associated pgf's.

\begin{Theorem}\label{thm:subcritical iid env}
Let $(Z_{n})_{n\ge 0}$ be subcritical and thus $R_{\infty}^{(-1)}<\infty=R_{\infty}$ a.s. Then the following assertions hold as $n\to\infty$ and for almost all realizations of the environment $\bfe$:
\begin{align}
&\Pi_{n}^{1/\theta}\,\bfP(Z_{n}>0)\ =\ \frac{1}{(1+R_{n}/\Pi_{n})^{1/\theta}}\ \idist\ \frac{1}{(1+R_{\infty}^{(-1)})^{1/\theta}},\label{eq:surv probab subcritical}\\
&\hspace{.2cm}\bfE(Z_{n}|Z_{n}>0)\ =\ (1+R_{n}/\Pi_{n})^{1/\theta}\ \idist\ (1+R_{\infty}^{(-1)})^{1/\theta},
\shortintertext{and}
&\hspace{.3cm}\left\|\bfP^{(n:1)}(Z_{n}\in\cdot|Z_{n}>0)-\PF_{+}\left(\theta,\frac{1}{1+R_{\infty}^{(-1)}}\right)\right\|\ \to\ 0.\label{eq:tv convergence subcritical}
\end{align}
\end{Theorem}

\begin{proof}
For the first two assertions it suffices to recall that $\Pi_{n}^{-1}R_{n}\eqdist R_{n}^{(-1)}$ for each $n$, and regarding \eqref{eq:tv convergence subcritical}, we have already pointed out above that
\begin{align*}
\bfP^{(n:1)}(Z_{n}\in\cdot|Z_{n}>0)\ =\ \PF_{+}\left(\theta,\frac{1}{1+R_{n}^{(-1)}}\right)\quad\text{a.s.}
\end{align*}
Since $R_{n}^{(-1)}\to R_{\infty}^{(-1)}$ a.s. and the considered distributions are supported by the same discrete set, the total variation convergence is immediate.\qed
\end{proof}

\begin{Theorem}\label{thm:critical iid env}
Let $(Z_{n})_{n\ge 0}$ be critical and thus $R_{\infty}=R_{\infty}^{(-1)}=\infty$ a.s. Then the following assertions hold as $n\to\infty$:
\begin{gather}
\big(\Pi_{n}R_{n}^{(-1)}\big)^{1/\theta}\,\bfP^{(n:1)}(Z_{n}>0)\ \to\ 1\quad\text{a.s.}\label{eq:surv probab critical}\\
\big(R_{n}^{(-1)}\big)^{1/\theta}\bfE^{(n:1)}(Z_{n}|Z_{n}>0)\ \to\ 1\quad\text{a.s.}\label{eq:surv mean critical}
\intertext{and for almost all realizations of the environment $\bfe$}
{\cblue\bfP^{(n:1)}\left[\frac{Z_{n}}{\big(R_{n}^{(-1)}\big)^{1/\theta}}\in\cdot\Bigg|Z_{n}>0\right]\ \idist\ \CPF_{+}(\theta,1).}\label{eq:asymp law critical}
\end{gather}
\end{Theorem}

\begin{proof}
Assertion \eqref{eq:surv probab critical} follows because
\begin{align*}
\big(\Pi_{n}R_{n}^{(-1)}\big)^{1/\theta}\,\bfP^{(n:1)}(Z_{n}>0)\,&=\,\left(\frac{\Pi_{n}R_{n}^{(-1)}}{\Pi_{n}+\Pi_{n}R_{n}^{(-1)}}\right)^{1/\theta}\,=\,\left(\frac{R_{n}^{(-1)}}{1+R_{n}^{(-1)}}\right)^{1/\theta}
\end{align*}
and $R_{n}^{(-1)}\to\infty$ a.s. The latter also provides \eqref{eq:surv mean critical} when combining it with $\bfE^{(n:1)}(Z_{n}|Z_{n}>0)=(1+R_{n}^{(-1)})^{1/\theta}$. Finally, if $\vph_{n}$ denotes the random Laplace transform of $(\Pi_{n}/R_{n})^{1/\theta}Z_{n}$ given $Z_{n}>0$ under $\bfP$, then
\begin{align*}
\frac{1}{(1-\vph_{n}(u))^{\theta}}\ &=\ \frac{1}{(1-h_{n}(\exp(-u(\Pi_{n}/R_{n})^{1/\theta}))^{\theta}}\\
&=\ \frac{1}{1+R_{n}/\Pi_{n}}\,\frac{1}{(1-\exp(-u(\Pi_{n}/R_{n})^{1/\theta}))^{\theta}}\,+\,\frac{R_{n}/\Pi_{n}}{1+R_{n}/\Pi_{n}}\\
&\to\ \frac{1}{u^{\theta}}\,+\,1\quad\text{a.s.~as }n\to\infty,
\end{align*}
and this proves \eqref{eq:asymp law critical} (see also \eqref{eq:CPFplus LT}).\qed
\end{proof}

\begin{Rem}\rm
The counterparts of \eqref{eq:surv probab critical} and \eqref{eq:surv mean critical} under $\bfP^{(1:n)}$ or $\bfP$ do only assert convergence in probability, viz.
\begin{align*}
R_{n}^{1/\theta}\,\bfP(Z_{n}>0)\ \iprob\ 1\quad\text{and}\quad\bigg(\frac{\Pi_{n}}{R_{n}}\bigg)^{1/\theta}\bfE(Z_{n}|Z_{n}>0)\ \iprob\ 1
\end{align*}
\end{Rem}

\section{Decomposition of supercritical processes}

Every supercritical \GWP\ $\cZ=(Z_{n})_{n\ge 0}$ with one ancestor and extinction probability $0<q<1$ can be decomposed into two nontrivial parts $\cZ_{1}=(Z_{1,n})_{n\ge 0}$ and $\cZ_{2}=(Z_{2,n})_{n\ge 0}$ by dividing each generation $n\ge 1$ into their individuals with finite progeny and those having an infinite line of descent, respectively, see \cite[p.\,47ff]{Athreya+Ney:72}. When conditioning upon extinction of $\cZ$, the process $\cZ_{1}$ then forms a subcritical \GWP\ with offspring pgf~$g(s)=q^{-1}f(qs)$, whereas $\cZ_{2}$ constitutes a supercritical \GWP\ with offspring pgf
$$ h(s)\,=\,\frac{f(q+(1-q)s)-q}{1-q}, $$
when conditioning upon explosion of $\cZ$. With $P$ denoting the transition kernel of the Markov chain $\cZ$, the law of $\cZ_{1}$, known as the Harris-Sevastyanov transform, is actually nothing but the Doob $h$-transform of $\cZ$ under the positive $P$-harmonic function $\N_{0}\ni i\mapsto q^{i}$, see \cite[Thm.~3.1]{KleRosSag:07}. It is also stated there that, if $f$ is linear-fractional, then $g$ is linear-fractional, see \cite[Prop.~3.1]{KleRosSag:07}, and the same holds true for $h$ as pointed out in \cite[Sect.\,4]{Alsmeyer:21}. The extension of these findings to power-fractional \GWP's in fixed environment, which naturally includes the linear-fractional ones $(\theta=1)$, is provided by the next result which has the interesting feature that the offspring law of the subcritical part $\cZ_{1}$ is again power-fractional, but with $\gamma>1$.

\begin{Theorem}\label{thm:decomposition supercritical GWP}
Let $\cZ=(Z_{n})_{n\ge 0}$ be a supercritical \GWP\ with offspring law $\PF(\theta,a,b)$ for $a\in (0,1)$ and $b>1-a$, thus $0<q<1$ by \eqref{eq:ext probab fixed environment}. Then
the offspring laws of $\cZ_{1}$ and $\cZ_{2}$, the latter when conditioned under $Z_{n}\to\infty$, are
$$ \PF(\theta,q^{-1},a,bq^{\theta})\quad\text{and}\quad\PF_{+}(\theta,a), $$
respectively.
\end{Theorem}

The extra notation for the power-fractional laws $\PF_{+}(\theta,a)$ on the \emph{positive} integers has been given here because they form the extension of their linear-fractional counterpart, i.e., the geometric laws $\Geom_{+}(a)$.

\begin{proof}
Put $\gamma=q^{-1}$. Using the definitions of $g$ and $h$ in terms of $f$ from above, the assertions follow from
\begin{align*}
\frac{1}{(\gamma-g(s))^{\theta}}\ &=\ \frac{q^{\theta}}{(1-f(qs))^{\theta}}\ =\ \frac{aq^{\theta}}{(1-qs)^{\theta}}\,+\,bq^{\theta}\ =\ \frac{a}{(\gamma-s)^{\theta}}\,+\,bq^{\theta}
\shortintertext{and}
\frac{1}{(1-h(s))^{\theta}}\ &=\ \frac{(1-q)^{\theta}}{(1-f(q+(1-q)s))^{\theta}}\ =\ \frac{a}{(1-s)^{\theta}}\,+\,b(1-q)^{\theta}\\
&=\ \frac{a}{(1-s)^{\theta}}\,+\,1-a,
\end{align*}
where $(1-q)^{\theta}=b^{-1}(1-a)$ has been utilized for the last equality.\qed
\end{proof}

In the random environment case, the situation becomes more complex. First, the extinction probabilities $q(\bfe_{\ssy\geq n}),\,n\ge 1,$ figuring in the pgf's of the quenched offspring laws of both $\cZ_{1}$ and $\cZ_{2}$ now depend on the entire future and vary with shifts in the environment. As a result, the environment of these processes changes from $\bfe$ to $(\bfe_{\ssy\geq n})_{n\ge 1}$ and is therefore no longer iid but only stationary and ergodic. This was previously noted in \cite{Alsmeyer:21} for the linear-fractional case. As a second complication, the quenched offspring laws of the subcritical part $\cZ_{1}$ are generally no longer of power-fractional type as described in (A1) or (A2). Instead, they fall under an extended interpretation, as will be discussed in Remark \ref{rem:ext PF}.

\begin{Theorem}\label{thm:decomposition supercritcal GWPRE}
Let $\cZ=(Z_{n})_{n\ge 0}$ be supercritical with $0<q(\bfe)<1$ a.s. Put also $q_{n}=q(\bfe_{\ssy\geq n})$ and $R_{n,\infty}=\Pi_{n}^{-1}(R_{\infty}-R_{n})$ for $n\ge 1$. Then the following assertions hold for $\cZ_{1}$ and $\cZ_{2}$ as introduced before:
\begin{itemize}\itemsep3pt
\item[(a)] Conditioned upon extinction of $\cZ$, i.e., under $\ovl\Prob:=\Prob(\cdot|Z_{n}\to 0)$, the process $\cZ_{1}$ is a subcritical \GWP\ in the stationary ergodic environment $(\bfe_{\ssy\geq n})_{n\ge 1}$ and with quenched offspring pgf's
\begin{gather}
g_{n}(s)\ =\ \frac{f_{n}(q_{n+1}s)}{q_{n}}\ =\ \frac{R_{n,\infty}^{1/\theta}}{R_{n,\infty}^{1/\theta}-1}f_{n}\left(\frac{(R_{n+1,\infty}^{1/\theta}-1)s}{R_{n+1,\infty}^{1/\theta}}\right)
\label{eq:pgf of Z_1}
\end{gather}
for $n\ge 1$.
\item[(b)] Conditioned upon survival of $\cZ$, i.e., under $\wh\Prob:=\Prob(\cdot|Z_{n}\to\infty)$, the process $\cZ_{2}$ is a nonextinctive power-fractional \GWP\ in the stationary ergodic environment $(\bfe_{\ssy\geq n})_{n\ge 1}$ and with quenched offspring laws
\begin{gather}
\PF_{+}\left(\theta,\frac{A_{n}R_{n+1,\infty}}{R_{n,\infty}}\right)\label{eq:laws of Z_2}
\shortintertext{and associated pgf's}
\begin{split}
h_{n}(s)\ &=\ \frac{f_{n}(q_{n+1}+(1-q_{n+1})s)-q_{n}}{1-q_{n}}\\
&=\ 1\,-\,\left[\frac{A_{n}R_{n+1,\infty}}{R_{n,\infty}(1-s)^{\theta}}+\frac{B_{n}}{R_{n,\infty}}\right]^{-1/\theta}
\end{split}
\label{eq:pgf of Z_2}
\end{gather}
for $n\ge 1$.
\end{itemize}
\end{Theorem}

\begin{proof}
We first note that $\bfe$, and thus also $(\bfe_{\ssy\geq n})_{n\ge 1}$ (see \cite[Prop.\,6.31]{Breiman:68}), remains stationary ergodic under both $\ovl{\Prob}$ and $\wh\Prob$. Namely,
\begin{align*}
&\ovl\Prob\big(\bfe_{i}\in E_{i},\,1\le i\le n\big)\ 
=\ \frac{1}{\Prob(Z_{n}\to 0)}\,\Erw\left[q(\bfe)\prod_{i=1}^{n}\1_{E_{i}}(\bfe_{i})\right]\\
&\hspace{1cm}=\ \frac{1}{\Prob(Z_{n}\to 0)}\,\Erw\left[q_{i}\prod_{i=1}^{n}\1_{E_{i}}(\bfe_{i+k})\right]\ =\ \ovl\Prob\big(\bfe_{i+k}\in E_{i},\,1\le i\le n\big)
\end{align*}
for any choice of measurable $E_{1},\ldots,E_{n}$ and all $n,k\in\N$ shows stationarity of $\bfe$ under $\ovl{\Prob}$, and the argument for $\wh\Prob$ is similar. The ergodicity follows because $0<q(\bfe)<1$ a.s.~ensures that the laws of $\bfe$ under $\Prob,\,\ovl{\Prob}$ and $\wh\Prob$ are equivalent (sharing the same null sets).

\vspace{.1cm}
The arguments showing that $g_{n}$ and $h_{n}$ are the pgf's of $Z_{1,n}$ and $Z_{2,n}$ under the quenched measures $\ovl\bfP:=\Prob(\cdot|Z_{n}\to 0,\bfe)$ and $\wh\bfP:=\Prob(\cdot|Z_{n}\to \infty,\bfe)$, respectively, are the same as for an ordinary Galton-Watson process, when additionally using that $q_{n}=f_{n}(q_{n+1})$ a.s.~for each $n\ge 1$. We also note that $q_{n}=(R_{n,\infty}^{1/\theta}-1)/R_{n,\infty}^{1/\theta}$ by \eqref{eq:ext probab varying environment} and the fact that $R_{n,\infty}$ is the copy of $R_{\infty}$ for the shifted environment $\bfe_{\ssy\geq n}$. The second equality in \eqref{eq:pgf of Z_1} is now immediate.

\vspace{.1cm}
Left with a proof of (b), it suffices to point out that
\begin{align*}
\frac{1}{(1-h_{n}(s))^{\theta}}\ &=\ \frac{(1-q_{n})^{\theta}}{(1-f_{n}(q_{n+1}+(1-q_{n+1})s))^{\theta}}\\
&=\ \frac{A_{n}(1-q_{n})^{\theta}}{(1-q_{n+1})^{\theta}(1-s)^{\theta}}\,+\,B_{n}(1-q_{n})^{\theta}\\
&=\ \frac{A_{n}R_{n+1,\infty}}{R_{n,\infty}(1-s)^{\theta}}\,+\,\frac{B_{n}}{R_{n,\infty}}
\end{align*}
for any $n\in\N$.\qed
\end{proof}

\begin{Rem}\label{rem:ext PF}\rm
One can easily check that \eqref{eq:pgf of Z_1} implies
\begin{equation*}
\frac{1}{q_{n}^{-1}-g_{n}(s)}\ =\ \frac{A_{n}(q_{n}/q_{n+1})^{\theta}}{(q_{n+1}^{-1}-s)^{\theta}}\,+\,B_{n}q_{n}^{\theta}
\end{equation*}
which is an equation of type \eqref{eq2:def PF gf}, but with different $\gamma$'s figuring on the left-hand and the sright-hand side, namely $q_{n}^{-1}$ and $q_{n+1}^{-1}$, respectively. On the other hand, writing $g_{n}$ in the form
\begin{gather*}
g_{n}(s)\,=\,\frac{1}{q_{n}}-\left[\frac{A_{n}(q_{n}/q_{n+1})^{\theta}}{(q_{n+1}^{-1}-s)^{\theta}}\,+\,B_{n}q_{n}^{\theta}\right]^{-1/\theta}
\shortintertext{we see that it differs only by a constant from}
\wh{g}_{n}(s)\,=\,\frac{1}{q_{n+1}}-\left[\frac{A_{n}(q_{n}/q_{n+1})^{\theta}}{(q_{n+1}^{-1}-s)^{\theta}}\,+\,B_{n}q_{n}^{\theta}\right]^{-1/\theta}\quad\text{for }s\in [0,1].
\end{gather*}
The latter function equals the pgf~of $\PF(\theta,q_{n+1}^{-1},A_{n}(q_{n}/q_{n+1})^{\theta},B_{n}q_{n}^{\theta})$, assuming Case (A2) is extended to allow  total mass greater than 1. This extension effectively removes the constraint $\frac{b}{1-a}\le (\gamma-1)^{-\theta}$, as detailed in Table \ref{tab:1.1}. In other words, for $\theta\in (0,1],\,\gamma>1$ and $a\in (0,1)$, the power-fractional measure $\PF(\theta,\gamma,a,b)$ is now defined by having pgf
\begin{equation*}
f(s)\ =\ \gamma\,-\,\left[\frac{a}{(\gamma-s)^{\theta}}\,+\,b\right]^{-1/\theta}
\end{equation*}
for any $b\ge\gamma^{-\theta}(1-a)$. Under this extension, we conclude in Theorem \ref{thm:decomposition supercritcal GWPRE} that the quenched law of $Z_{1,n}$ given extinction of $\cZ$ equals
\begin{gather}\label{eq:quenched law of Z_{1,n}}
\left(\frac{1}{q_{n}}-\frac{1}{q_{n+1}}\right)\delta_{0}\,+\,\PF\left(\theta,\frac{1}{q_{n+1}},\frac{A_{n}q_{n}^{\theta}}{q_{n+1}^{\theta}},B_{n}q_{n}^{\theta}\right)\quad\text{a.s.}
\end{gather}
for each $n\in\N$, where the additional mass at 0 can be positive or negative, and also vanish.
\end{Rem}


\section{Power-fractional laws in continuous-time branching}

In this section, we briefly discuss the question of whether power-fractional branching processes in continuous time can be defined, but limit ourselves to the class of Markov branching processes. This has already been done in greater detail in \cite{LindoSagitZhum:23}, including the case of varying environment.

\vspace{.1cm}
The \emph{linear birth-death process} $(Y(t))_{t\ge 0}$ is not only a simple example of a Markov branching process, but also the one whose one-dimensional marginals are all linear-fractional. Specifically, assume that individuals, acting independently, give birth to a new individual at rate $\lambda>0$ and die at rate $\mu\ge 0$. This means that each individual has an exponential lifetime with parameter $\lambda+\mu$, and at the end of its life either gives birth to two children, which occurs with probability $\lambda/(\lambda+\mu)$, or dies with probability $\mu/(\lambda+\mu)$. Thus, the process has offspring distribution $\frac{\mu}{\lambda+\mu}\delta_{0}+\frac{\lambda}{\lambda+\mu}\delta_{2}$, with associated pgf
\begin{equation*}
g(s)\,:=\,\frac{\mu}{\lambda+\mu}\,+\,\frac{\lambda}{\lambda+\mu}s^{2},\quad s\in [0,1].
\end{equation*}
Put $\rho:=\lambda-\mu$, 
\begin{gather*}
\alpha(t)\ =\ \begin{cases}
\displaystyle\frac{\mu e^{\rho t}-\mu}{\lambda e^{\rho t}-\mu}&\text{if }\lambda\ne\mu\\[3mm]
\displaystyle\frac{\lambda t}{\lambda t+1}&\text{if }\lambda=\mu
\end{cases},\quad\text{and}\quad\beta(t)\ =\ \frac{\lambda}{\mu}\alpha(t)
\end{gather*}
for $t\ge 0$. Let 
$$ G(s,t)\,:=\,\Erw(s^{Y(t)}|Y(0)=1)\quad\text{for }(s,t)\in [0,1]\times [0,\infty) $$ 
denote the pgf of $Y(t)$ given $Y(0)=1$. Then the branching property implies
\begin{gather}\label{eq:branching property G}
G(s,t+u)\ =\ G(G(s,t),u)\quad\text{for all }(s,t,u)\in [0,1]\times [0,\infty).
\end{gather}
Moreover, with the boundary condition $G(s,0)=s$, the bivariate function $G(s,t)$ forms the unique solution to the forward and backward equation
\begin{gather*}
\frac{\partial}{\partial t}G(s,t)\ =\ (\lambda+\mu)(g(s)-s)\,\frac{\partial}{\partial s}G(s,t)
\shortintertext{and}
\frac{\partial}{\partial t}G(s,t)\ =\ g(G(s,t))-G(s,t)
\end{gather*}
respectively, in particular
\begin{equation}\label{eq:def pgf linear-fractional MBP}
(\lambda+\mu)(g(s)-s)\,=\,\frac{\partial}{\partial t}G(s,0+)\,:=\,\lim_{t\downarrow 0}\frac{G(s,t)-s}{t},\quad s\in[0,1].
\end{equation}
These equations can be solved explicitly, yielding
\begin{gather}\label{eq:pgf LDBP}
G(s,t)\ =\ \alpha(t)\ +\ (1-\alpha(t))\frac{(1-\beta(t))s}{1-\beta(t)s},
\intertext{where, setting $\rho:=\lambda-\mu$,}
\alpha(t)\ =\ \begin{cases}
\displaystyle\frac{\mu e^{\rho t}-\mu}{\lambda e^{\rho t}-\mu}&\text{if }\lambda\ne\mu\\[3mm]
\displaystyle\frac{\lambda t}{\lambda t+1}&\text{if }\lambda=\mu
\end{cases},\quad\text{and}\quad\beta(t)\ =\ \begin{cases}
\displaystyle\frac{\lambda e^{\rho t}-\lambda}{\lambda e^{\rho t}-\mu}&\text{if }\lambda\ne\mu\\[3mm]
\displaystyle\frac{\lambda t}{\lambda t+1}&\text{if }\lambda=\mu
\end{cases},\nonumber
\end{gather}
see e.g.~\cite{Tavare:18}. The uniqueness of the solution ensures that the linear birth-death process $(Y(t))_{t\ge 0}$ is the unique Markov branching process with pgf's $G(\cdot,t)$ for $t\ge 0$, see \cite[Thm.\,III.2.1]{Athreya+Ney:72}. From \eqref{eq:pgf LDBP}, it follows that all $Y(t)$ are linear-fractional, specifically
\begin{gather}
\begin{split}
Y(t)\ &\eqdist\ \alpha(t)\delta_{0}\,+\,(1-\alpha(t))\,\Geom_{+}(1-\beta(t))\\
&=\ \begin{cases}
\LF\bigg(e^{-\rho t},\frac{\lambda}{\rho}(1- e^{-\rho t})\bigg)&\text{if }\lambda\ne\mu,\\
\LF(1,\lambda t)&\text{if }\lambda=\mu.
\end{cases}
\end{split}
\end{gather}
In the pure birth case $\mu=0<\lambda$, the well-known \emph{Yule process} is obtained, where
\begin{equation*}
Y(t)\ \eqdist\ \LF(e^{-\lambda t},1-e^{-\lambda t})\ =\ \Geom_{+}(e^{-\lambda t}),\quad t\ge 0.
\end{equation*}

We will now demonstrate that Markov branching processes with power-frac\-tional marginals $\PF(\theta,a(t),b(t))$ for some $\theta$ less than $1$ -- specifically, their pgf's -- can be derived from $G(s,t)$ (which is representing the case $\theta=1$) by utilizing the connection between linear-fractional and power-fractional laws through conjugation with a Sibuya law, as described in (f) of the Introduction (see \eqref{eq:conjugation rule with Sibuya}). 

\vspace{.1cm}
To this end, let $h$ be the pgf of the Sibuya distribution $\Sib(\theta)$ for some fixed $\theta\in (0,1)$, and let $(Y(t))_{t\ge 0}$ be  the linear birth-death process defined above with individual birth and death rates $\mu$ and $\lambda$, respectively, and with pgf's $G(\cdot,t)$ given by \eqref{eq:pgf LDBP}. For each $t\ge 0$, define the pgf $F(\cdot,t)$ as the conjugation of $G(\cdot,t)$ with $h$, that is
\begin{equation}\label{eq:F conjugation}
F(s,t)\,;=\,h^{-1}(G(h(s),t)\quad\text{for }(s,t)\in [0,1]\times [0,\infty)
\end{equation}
where $h^{-1}$ denotes the inverse of $h$ on $[0,1]$ and can easily be computed as
$$ h^{-1}(s)\ =\ 1-(1-s)^{1/\theta},\quad s\in [0,1]. $$
It is straightforward to verify that \eqref{eq:branching property G} carries over to $F(s,t)$, thus
\begin{gather*}
F(s,t+u)\ =\ F(F(s,t),u)\quad\text{for all }(s,t,u)\in [0,1]\times [0,\infty).
\end{gather*}
Furthermore, with the initial condition $F(s,0)=s$ and for some $\nu>0$ and some pgf $f$, the function $F(s,t)$ must be the unique solution to the forward and backward equation
\begin{gather*}
\frac{\partial}{\partial t}F(s,t)\ =\ \nu(f(s)-s)\,\frac{\partial}{\partial s}F(s,t)
\shortintertext{and}
\frac{\partial}{\partial t}F(s,t)\ =\ f(F(s,t))-F(s,t)
\end{gather*}
respectively. The function $f$ is derived as follows using \eqref{eq:def pgf linear-fractional MBP}, the relation \eqref{eq:F conjugation}, and the constraint that $f'(0)$, which represents the probability for an individual to have exactly one child, must vanish:
\begin{gather*}
\nu(f(s)-s)\ =\ \frac{\partial}{\partial t}F(s,0+)\ =\ (h^{-1})'(G(h(s),0))\frac{\partial}{\partial t}G(h(s),0+)
\shortintertext{which simplifies to}
f(s)-s\ =\ \frac{(\lambda+\mu)(h^{-1})'(h(s))(g(h(s))-h(s))}{\nu}\ =\ \frac{(\lambda+\mu)(g(h(s))-h(s))}{\nu h'(s)}.
\intertext{With explicit expressions for $g$ and $h$, this yields}
f(s)\ =\ 1\,-\,\frac{\theta\nu+\rho}{\theta\nu}(1-s)\,+\,\frac{\lambda}{\theta\nu}(1-s)^{1+\theta}.
\end{gather*}
Finally, $f'(0)=(\theta\nu)^{-1}(\theta\nu+\rho-\lambda(1+\theta))=0$ provides
\begin{equation}\label{eq:parameter constraint}
\nu\ =\ \lambda+\mu\theta^{-1},
\end{equation}
in particular $\nu\ge\lambda$. Thus, with $\sfm:=(\theta\nu+\rho)/\theta\nu=\lambda(1+\theta)/\theta\nu$, we arrive at
\begin{equation}\label{eq:final form of f MBP}
f(s)\ =\ 1\,-\,\sfm(1-s)\,+\,\frac{\sfm}{1+\theta}(1-s)^{1+\theta},
\end{equation}
which is the same form of $f$ as stated in \cite[Eq.\,(3)]{LindoSagitZhum:23}. We note that $\sfm=f'(1)$, so it represents the mean of the distribution associated with $f$.

\vspace{.1cm}
After these considerations, we can state the following result:

\begin{Theorem}\label{thm:power-fractional MBP}
Given any $\theta\in (0,1)$ and $\nu>0$, let $\lambda>0$ and $\mu\ge 0$ be such that \eqref{eq:parameter constraint} holds. Let $(Y(t))_{t\ge 0}$ be a linear birth-death process with individual birth rate~$\lambda$, death rate $\mu$ and pgf's $G(\cdot,t)$, $t\ge 0$. Then there exists a Markov~branching process $(Z(t))_{t\ge 0}$ with pgf's $F(\cdot,t)$ for $t\ge 0$, which has the following properties: 
\begin{itemize}\itemsep2pt
\item[(i)] Individuals in the population described by the process have exponential lifetimes with parameter $\nu$ and offspring law with pgf $f$ defined in \eqref{eq:final form of f MBP}.
\item[(ii)] $F(s,t)$ is the conjugation of $G(s,t)$ stated in \eqref{eq:F conjugation} with the pgf of the Sibuya law $\Sib(\theta)$.
\item[(iii)] For each $t>0$, the law of $Z(t)$ is power-fractional, specifically
\begin{equation}\label{eq:law of power-fractional MBP Z(t)}
Z(t)\ \eqdist\ \PF\big(\theta,e^{-\rho t},\lambda\rho^{-1}(1-e^{-\rho t}\big)
\end{equation}
\end{itemize}
\end{Theorem}

\begin{proof}
For (i) and (ii), it suffices to refer to the discussion preceding the theorem. For (iii), we refer to
\cite{LindoSagitZhum:23}, where it is shown that 
\begin{align*}
F(s,t)\ =\ 1\,-\,\bigg(e^{-(\sfm-1)\nu\theta t}(1-s)^{-\theta}\,+\,\frac{\sfm\theta\nu}{1+\theta}\int_{0}^{t}e^{-(\sfm-1)\nu u}\,du\bigg)^{-1/\theta}
\end{align*}
for all $(s,t)\in [0,1]\times [0,\infty)$. Consequently, the law of each $Z(t)$, $t>0$, is power-fractional, namely
\begin{equation*}
Z(t)\ \eqdist\ \PF\bigg(\theta,e^{-(\sfm-1)\theta\nu t},\frac{\sfm\nu\theta}{\theta+1}\int_{0}^{t}e^{-(\sfm-1)\theta\nu u}\,du\bigg).
\end{equation*}
Finally, this is the same as \eqref{eq:law of power-fractional MBP Z(t)} because, by \eqref{eq:parameter constraint} and the definition of $\sfm$, we have $\sfm\nu\theta=\lambda(\theta+1)$ and $\sfm-1=\rho/\theta\nu$.\qed
\end{proof}

\begin{Rem}\label{rem:power-fractional Yule}\rm
If $(Y(t))_{t\ge 0}$ is chosen as a Yule process, meaning the individual death rate $\mu=0$, it is easily verified that $f(0)=0$. This can be understood intuitively because, in this case, the Markov branching process $(Z(t))_{t\ge 0}$ (in terms of its pgf's $F(\cdot,t)$) is the conjugation of a pure birth process with a distribution that assigns no probability mass at zero. As an immediate  consequence $F(0,t)=0$, which in turn entails that the individual offspring distribution also has no mass at 0. Since $\mu=0$ implies $\rho=\lambda$, we see that \eqref{eq:law of power-fractional MBP Z(t)} simplifies to
\begin{equation}\label{eq:law of power-fractional Yule Z(t)}
Z(t)\ \eqdist\ \PF(\theta,e^{-\lambda t},1-e^{-\lambda t})\ =\ \PF_{+}(\theta,e^{-\lambda t})
\end{equation}
for each $t>0$. We have thus obtained a power-fractional analog of the Yule process.
\end{Rem}

\begin{Rem}\rm
Let $(p_{n})_{n\ge 0}$ denote the probability distribution associated with $f$. Then it follows directly from \eqref{eq:final form of f MBP} that
\begin{equation*}
\frac{1-f(s)}{\sfm(1-s)}\ =\ 1\ -\ \frac{1}{1+\theta}(1-s)^{\theta}\ =\ \frac{1}{\sfm}\sum_{n\ge 0}\Bigg(\sum_{k>n}p_{k}\Bigg)s^{n}.
\end{equation*}
This shows that the normalized \emph{tail measure} associated with $(p_{n})_{n\ge 0}$, given by $(\sfm^{-1}\sum_{k>n}p_{k})_{n\ge 0}$, is the generalized Sibuya distribution $\GSib(\theta,(1+\theta)^{-1})$. Since the probability mass function of this law is known (see \eqref{eq:pmf Sib}), the $p_{n}$'s can be derived from it, resulting in
\begin{equation*}
p_{n}\ =\ \begin{cases} 1-\sfm\theta(1+\theta)^{-1}&\text{if }n=0,\\ 0&\text{if }n=1,\\ \sfm\theta/2&\text{if }n=2,\\
\sfm\theta(1-\theta)\cdots (n-2-\theta)/n!&\text{if }n\ge 3.
\end{cases}
\end{equation*}
Again, this is already stated in \cite{LindoSagitZhum:23}.
\end{Rem}


\section{Auxiliary lemmata about power-fractional laws}

\begin{Lemma}
Let $f$ be the pgf~of $\PF(\theta,\gamma,a,b)$ for $\theta\in [-1,1]\backslash\{0\}$, $\gamma>1$, and $a\in (0,1)$ (see Table \ref{tab:1.1}). Then $f_{\gamma}(s):=f(\gamma s)/\gamma$ is the pgf~of the proper supercritical power-fractional distribution $\PF(\theta,1,a,b\gamma^{\theta})$, which has a unique fixed point at $1/\gamma$ in the interval $(0,1)$.
\end{Lemma}

\begin{proof}
It follows from Eq.\,\eqref{eq2:def PF gf} upon multiplication with $\gamma^{\theta}$ that
\begin{align*}
\frac{1}{(1-\gamma^{-1}f(\gamma s))^{\theta}}\ =\ \frac{a}{(1-s)^{\theta}}\,+\,b\gamma^{\theta}
\end{align*}
and this obviously implies the assertions.\qed
\end{proof}

\begin{Lemma}\label{lem:conditional law PF}
Let $X$ be a random variable with law $\PF(\theta,a,b)$ for $\theta\in (0,1]$. Then the conditional law of $X$ given $X>0$ is again power-fractional, namely
\begin{equation*}
\Prob(X\in\cdot|X>0)\ =\ \PF\left(\theta,\frac{a}{a+b},\frac{b}{a+b}\right).
\end{equation*}
\end{Lemma}

\begin{proof}
Let $f$ denote the pgf of $X$. Then the pgf~of the conditional law in question is given by
\begin{gather*}
h(s)\ =\ \frac{f(s)-f(0)}{1-f(0)}\ =\ 1\,-\,\frac{1-f(s)}{1-f(0)},
\shortintertext{hence}
\frac{1}{(1-h(s))^{\theta}}\ =\ \left(\frac{1-f(0)}{1-f(s)}\right)^{\theta}\ =\ \frac{1}{a+b}\left[\frac{a}{(1-s)^{\theta}}+b\right],
\end{gather*}
which provides the asserted result.\qed
\end{proof}

It is a well-known fact that when \emph{thinning} a Poisson-distributed set of points by independently tossing a $p$-coin for each point, the number of remaining points follows again a Poisson distribution. More generally, this also holds for the number of points with a given label when each point is \emph{tagged} independently by assigning a label $i$ with probability $p_{i}$, where $i$ belongs to a finite or countably infinite set of labels $\cL$ and $\sum_{i\in\cL}p_{i}=1$. The following lemma shows that this property extends to power-fractional distributions.

\begin{Lemma}[Thinning/Tagging]\label{lem:PF sum of Bernoullis}
Let $X,I_{1},I_{2},\ldots$ be independent random va\-riables, where $X\eqdist\PF(\theta,a,b)$ for $\theta\in (0,1]$, and the $I_{n}$ take values in a finite or countably infinite set of labels $\cL$, with common law $(\gamma_{i})_{i\in\cL}$. Then the law of the random sum $Z(i):=\sum_{k=1}^{X}\1_{\{I_{k}=i\}}$ is also power-fractional, namely $\PF(\theta,a\gamma_{i}^{-\theta},b)$. for each $i\in\cL$.
\end{Lemma}

\begin{proof}
Since the law of $\1_{\{I_{k}=i\}}$ is Bernoulli with parameter $\gamma_{i}$ and has the pgf $g_{i}(s)=\gamma_{i}s+1-\gamma_{i}$, we infer that $Z(i)$ has the pgf $f\circ g_{i}$, which satisfies
\begin{align*}
\frac{1}{(1-f(g_{i}(s)))^{\theta}}\ =\ \frac{a}{(1-g_{i}(s))^{\theta}}\,+\,b\ =\ \frac{a\gamma_{i}^{-\theta}}{(1-s)^{\theta}}\,+\,b.
\end{align*}
Hence, $Z(i)$ has the asserted power-fractional law.\qed
\end{proof}

The lemma gives rise to some additional observations that finally result in the definition of \emph{$m$-dimensional power-fractional distributions} for $m\in\N\cup\{\infty\}$. 
\begin{itemize}\itemsep3pt
\item[(a)] One can combine any subset $B$ of labels from the set $\cL$ into a single label, yielding
$$ Z(B)\,:=\,\sum_{i\in B}Z(i)\,=\,\sum_{k=1}^{X}\1_{B}(I_{k}), $$
and it can be readily verified that
\begin{equation}\label{eq:law of Z(B)}
Z(B)\ \eqdist\ \PF\bigg(\theta,\frac{a}{\gamma(B)^{\theta}},b\bigg),
\end{equation}
where $\gamma(B):=\sum_{i\in B}\gamma_{i}$.
\item[(b)] One can interpret $Z$ as a random point measure, or equivalently, as a point process on the label set $\cL$, by writing $Z=\sum_{k=1}^{X}\delta_{I_{k}}$. As stated in \eqref{eq:law of Z(B)}, $Z(B)$ has a power-fractional distribution for every $B\subset\cL$. However, unlike the Poisson case, the random variables $Z(B_{1}),\ldots,Z(B_{n})$ for pairwise disjoint subsets $B_{i}$ are generally not independent. In particular, the $Z(i)$ for $i\in\cL$ are typically not independent.
\item[(c)] Finally, we note that the conditional law of $(Z(i))_{i\in\cL}$ given $Z(\cL)=z\in\N$ is multinomial with parameters $z$ and $\gamma_{i},\,i\in\cL$, i.e.,
\begin{equation*}
\Prob(Z(i)=z_{i},\,i\in\cL|Z(\cL)=z)\ =\ \frac{z!}{\prod_{i\in\cL}z_{i}!}\prod_{i\in\cL}\gamma_{i}^{z_{i}}
\end{equation*}
for each $(z_{i})_{i\in\cL}\in\N_{0}^{\cL}$ such that $\sum_{i\in\cL}z_{i}=z$.
\end{itemize}

After these observations, the unique law of the random vector $(Z(i))_{i\in\cL}$ can be referred to as the \emph{$m$-dimensional power-fractional distribution with parameters $\theta,a,b$ and $(\gamma_{i})_{i\in\cL}$}, abbreviated as $\PF_{m}(\theta,a,b,(\gamma_{i})_{i\in\cL})$, where $m=|\cL|$. It is of particular interest in connection with multitype power-fractional \GWP's that will be discussed in a future paper.


\nocite{LindoSagitZhum:23}
\nocite{SagitovZhumayev:24}

\section*{\ackname}
Gerold Alsmeyer acknowledges funding from the German Science Foundation (DFG) under Germany's Excellence Strategy EXC 2044--390685587, Mathematics M\"unster: Dynamics--Geometry--Structure. He also wishes to thank the organizers of the conference on branching processes and their applications in Badajoz (Spain) in April 2024 for providing the opportunity to present a significant portion of this work. Last but not least, he would like to express his gratitude to Bastien Mallein and Serik Sagitov for insightful comments and discussions on the subject.
Viet Hung Hoang gratefully acknowledges funding from the European Research Council (ERC) under the European Union's Horizon 2020 research and innovation programme (Grant Agreement No. 759702) and from the University of M\"unster. He also thanks the Industrial University of Ho Chi Minh City for providing a supportive environment to complete the project.

%
%

\bibliographystyle{abbrv}
\bibliography{StoPro}

\def\cprime{$'$}
\begin{thebibliography}{10}

\bibitem{Alsmeyer:21}
G.~Alsmeyer.
\newblock Linear fractional {G}alton-{W}atson processes in random environment
  and perpetuities.
\newblock {\em Stoch. Qual. Control}, 36(2):111--127, 2021.

\bibitem{AlsIksRoe:09}
G.~Alsmeyer, A.~Iksanov, and U.~R{\"o}sler.
\newblock On distributional properties of perpetuities.
\newblock {\em J. Theoret. Probab.}, 22(3):666--682, 2009.

\bibitem{AthreyaKarlin:71a}
K.~B. Athreya and S.~Karlin.
\newblock On branching processes with random environments, {I}: Extinction
  probabilities.
\newblock {\em Ann. Math. Stat.}, 42(5):pp. 1499--1520, 1971.

\bibitem{Athreya+Ney:72}
K.~B. Athreya and P.~E. Ney.
\newblock {\em Branching processes}.
\newblock Die Grund\-lehren der mathematischen Wissenschaften, Band 196.
  Springer, New York, 1972.

\bibitem{BingGolTeug:89}
N.~H. Bingham, C.~M. Goldie, and J.~L. Teugels.
\newblock {\em Regular variation}, volume~27 of {\em Encyclopedia of
  Mathematics and its Applications}.
\newblock Cambridge University Press, Cambridge, 1989.

\bibitem{Breiman:68}
L.~Breiman.
\newblock {\em Probability}.
\newblock Addison-Wesley Publishing Company, Reading, Mass., 1968.

\bibitem{BurDamMik:16}
D.~Buraczewski, E.~Damek, and T.~Mikosch.
\newblock {\em Stochastic models with power-law tails}.
\newblock Springer Series in Operations Research and Financial Engineering.
  Springer, [Cham], 2016.
\newblock The equation $X=AX+B$.

\bibitem{GolMal:00}
C.~M. Goldie and R.~A. Maller.
\newblock Stability of perpetuities.
\newblock {\em Ann. Probab.}, 28(3):1195--1218, 2000.

\bibitem{Iksanov:16}
A.~Iksanov.
\newblock {\em Renewal theory for perturbed random walks and similar
  processes}.
\newblock Probability and its Applications. Birkh\"auser/Springer, Cham, 2016.

\bibitem{KerstingVatutin:17}
G.~Kersting and V.~Vatutin.
\newblock {\em Discrete Time Branching Processes in Random Environment}.
\newblock Wiley-ISTE, London and Hoboken, 2017.
\newblock Mathematics and Statistics Series.

\bibitem{Kesten:73}
H.~Kesten.
\newblock Random difference equations and renewal theory for products of random
  matrices.
\newblock {\em Acta Math.}, 131:207--248, 1973.

\bibitem{KleRosSag:07}
F.~C. Klebaner, U.~R\"{o}sler, and S.~Sagitov.
\newblock Transformations of {G}alton-{W}atson processes and linear fractional
  reproduction.
\newblock {\em Adv. in Appl. Probab.}, 39(4):1036--1053, 2007.

\bibitem{LindoSagitov:15}
A.~Lindo and S.~Sagitov.
\newblock Probability generating functions with explicit iterations, 2015.
\newblock Preprint.

\bibitem{SagitovLindo:16}
S.~Sagitov and A.~Lindo.
\newblock A special family of {G}alton-{W}atson processes with explosions.
\newblock In {\em Branching processes and their applications}, volume 219 of
  {\em Lect. Notes Stat.}, pages 237--254. Springer, Cham, 2016.

\bibitem{LindoSagitZhum:23}
S.~Sagitov, A.~Lindo, and Y.~Zhumayev.
\newblock Theta-positive branching in varying environment, 2023.
\newblock Preprint available at {\tt https://arxiv.org/abs/2303.04230},.

\bibitem{SagitovZhumayev:24}
S.~Sagitov and Y.~Zhumayev.
\newblock Galton-{W}atson theta-processes in a varying environment.
\newblock {\em Stochastics and Quality Control}, 2024.

\bibitem{Sibuya:79}
M.~Sibuya.
\newblock Generalized hypergeometric, digamma and trigamma distributions.
\newblock {\em Ann. Inst. Statist. Math.}, 31(3):373--390, 1979.

\bibitem{Tavare:18}
S.~Tavar\'e.
\newblock The linear birth-death process: an inferential retrospective.
\newblock {\em Adv. in Appl. Probab.}, 50(A):253--269, 2018.

\bibitem{Vervaat:79}
W.~Vervaat.
\newblock On a stochastic difference equation and a representation of
  nonnegative infinitely divisible random variables.
\newblock {\em Adv. in Appl. Probab.}, 11(4):750--783, 1979.

\end{thebibliography}


\begin{thebibliography}{6}
%

\bibitem {smit:wat}
Smith, T.F., Waterman, M.S.: Identification of common molecular subsequences.
J. Mol. Biol. 147, 195?197 (1981). \url{doi:10.1016/0022-2836(81)90087-5}

\bibitem {may:ehr:stein}
May, P., Ehrlich, H.-C., Steinke, T.: ZIB structure prediction pipeline:
composing a complex biological workflow through web services.
In: Nagel, W.E., Walter, W.V., Lehner, W. (eds.) Euro-Par 2006.
LNCS, vol. 4128, pp. 1148?1158. Springer, Heidelberg (2006).
\url{doi:10.1007/11823285_121}

\bibitem {fost:kes}
Foster, I., Kesselman, C.: The Grid: Blueprint for a New Computing Infrastructure.
Morgan Kaufmann, San Francisco (1999)

\bibitem {czaj:fitz}
Czajkowski, K., Fitzgerald, S., Foster, I., Kesselman, C.: Grid information services
for distributed resource sharing. In: 10th IEEE International Symposium
on High Performance Distributed Computing, pp. 181?184. IEEE Press, New York (2001).
\url{doi: 10.1109/HPDC.2001.945188}

\bibitem {fo:kes:nic:tue}
Foster, I., Kesselman, C., Nick, J., Tuecke, S.: The physiology of the grid: an open grid services architecture for distributed systems integration. Technical report, Global Grid
Forum (2002)

\bibitem {onlyurl}
National Center for Biotechnology Information. \url{http://www.ncbi.nlm.nih.gov}


\end{thebibliography}

\end{document}